\documentclass[12pt]{amsart}

\usepackage{amsmath,amsthm,amscd,amsfonts,amssymb,epic,eepic,bbm}
\usepackage[pagebackref,colorlinks=true,linkcolor=blue,citecolor=blue]{hyperref}

\allowdisplaybreaks
\newtheorem{thm}{Theorem}[section]

\newtheorem{lem}[thm]{Lemma}

\newtheorem{prop}[thm]{Proposition}
\theoremstyle{remark}

\newcounter{remarkscounter}

\numberwithin{equation}{section}
\newcommand{\A}{\mathbb{A}}
\newcommand{\GL}{\mathrm{GL}}
\newcommand{\gl}{\mathfrak{gl}}

\newcommand{\ZZ}{\mathbb{Z}}

\newcommand{\QQ}{\mathbb{Q}}

\newcommand{\lto}{\longrightarrow}

\newcommand{\OO}{\mathcal{O}}
\newcommand{\CC}{\mathbb{C}}
\newcommand{\RR}{\mathbb{R}}
\newcommand{\GG}{\mathbb{G}}

\newcommand{\quash}[1]{}

\theoremstyle{definition}

\renewcommand{\bar}{\overline}

\numberwithin{equation}{subsection}

\newcommand{\one}{\mathbbm{1}}

\begin{document}

\title[Secondary terms for quadratic forms]{Secondary terms in asymptotics for the number of zeros of quadratic forms over number fields}
\author{Jayce R. Getz}
\address{Department of Mathematics\\
Duke University\\
Durham, NC 27708}
\email{jgetz@math.duke.edu}


\thanks{The author is thankful for partial support provided by NSF grant DMS-1405708.  Any opinions, findings, and conclusions or recommendations expressed in this material are those of the author and do not necessarily reflect the views of the National Science Foundation.}
\maketitle

\begin{abstract}

Let $Q$ be a nondegenerate quadratic form on a vector space $V$ of even dimension $n$ over a number field $F$.   Via the circle method or automorphic methods one can give good estimates for smoothed sums over the number of zeros of the quadratic form whose coordinates are of size at most $X$ (properly interpreted).  For example, when $F=\QQ$ and $\dim V>4$ Heath-Brown has given an asymptotic of the form 
\begin{align} \label{HB:esti}
c_1X^{n-2}+O_{Q,\varepsilon,f}(X^{n/2+\varepsilon})
\end{align}
for any $\varepsilon>0$.  Here $c_1 \in \CC$ and $f \in \mathcal{S}(V(\RR))$ is a smoothing function.
We refine Heath-Brown's work to give an asymptotic of the form
$$
c_1X^{n-2}+c_2X^{n/2}+O_{Q,\varepsilon,f}(X^{n/2+\varepsilon-1})
$$
over any number field.  Here $c_2 \in \CC$.  
Interestingly the secondary term $c_2$ is the sum of a rapidly decreasing function on $V(\RR)$ over the zeros of $Q^{\vee}$, the form whose matrix is inverse to the matrix of $Q$.  We also prove analogous results in the boundary case $n=4$, generalizing and refining Heath-Brown's work in the case $F=\QQ$.
\end{abstract}

\tableofcontents

\section{Introduction}

\subsection{The circle method applied to quadratic forms}
Let $V=\mathbb{G}_a^n$, thus $V(\QQ)$ is a vector space of dimension $n$ over $\QQ$.  Let $Q$ be a quadratic form on $V(\QQ)$.  Let
$$
\Delta:\RR_{>0} \lto V(\RR)
$$
be the diagonal embedding and let $f_\infty \in \mathcal{S}(V(\RR))$ (the usual Schwartz space).  A classical problem in analytic number theory is the asymptotic evaluation of
\begin{align} \label{gamma:sum}
N(X):=N(X,Q,f_\infty):=\sum_{\substack{\xi \in V(\ZZ)\\Q(\xi) =0}}f_\infty\left(\frac{\xi}{\Delta(X)} \right).
\end{align}
This is a smoothed version of the number of zeros of the quadratic form of height at most $X$.

Suppose $n>4$.  The circle method (or Kloosterman's refinement of it) allows one to establish an asymptotic of the form 
\begin{align} \label{N(X):esti}
N(X)=c_1X^{n-2}+o_{Q,f_\infty}(X^{n-2})
\end{align}
where $c_1 \in \CC$ is essentially the so-called singular series.  In the beautiful paper \cite{HBnewcircle}
Heath-Brown  refined the error term in \eqref{N(X):esti} and extended the range in which one could obtain similar asymptotics to $n>2$. For example, in the range $n>4$ he proved that for any $\varepsilon>0$ one has
\begin{align} \label{HB:esti}
c_1X^{n-2}+O_{Q,f_\infty,\varepsilon}(X^{(n-1+\beta)/2+\varepsilon})
\end{align}
where $\beta \in \{0,1\}$ is $0$ if $n$ is odd and $1$ if $n$ is even 
(see \cite[Theorem 5]{HBnewcircle}).  Heath-Brown's primary tool was
an idea of Duke, Friedlander, and Iwaniec \cite{DFI} that will be mentioned again later.

In this paper we refine this asymptotic still further. For even $n>4$ we obtain an asymptotic of the form
\begin{align}
N(X)=c_1X^{n-2}+c_2X^{n/2}+O_{\varepsilon}(X^{n/2-1+\varepsilon})
\end{align}
for some $c_2 \in \CC$. The complex number $c_2$ can be nonzero (see 
\cite{Getz:RSMonoid}) so this asymptotic implies in particular that Heath-Brown's 
estimate \eqref{HB:esti} is essentially sharp (at least for $n$ even).  For $n=4$ we also obtain an asymptotic, which is a refinement of Heath-Brown's estimate in this case.  In fact refining Heath-Brown's work on the boundary case of $n=4$ is the key to our argument.  We will explain this in more detail in \S \ref{ssec:sketch}.

\subsection{Statement of results over the rationals}  \label{ssec:statement} The main theorem of this paper, Theorem \ref{thm:main}, is valid over any number field.    We work adelically because working classically unnecessarily complicates matters (especially over number fields).  It also allows one to treat congruence conditions on the zeros of the quadratic form with no additional effort.  
 To ease the path for readers unfamiliar with this language, we will only state our results over $\QQ$ in the introduction.  We could work entirely classically in this context, but we feel that working adelically and translating the work explicitly back to classical language will help the reader tackle the rest of the paper.  In the introduction we will also only use an unramified test function at the finite places; this corresponds to having no supplementary congruence conditions on the zeros of $Q$ over which we are summing.

To state our result, let $\Phi_{0\infty} \in \mathcal{S}(\RR \times \RR)$, the Schwartz space of $\RR \times \RR$.  We assume that $\Phi_{0\infty}(t,0)=0$ for all $t$ and that $\int_{\RR}\Phi(0,t)dt=1$.
Let $\A_\QQ$ denote the adeles of $\QQ$, and let $\psi:\QQ \backslash \A_\QQ \to \CC^\times$ be a nontrivial character.  For ease of exposition, in the introduction we will take the character
\begin{align}
\psi=\psi_\infty \prod_p\psi_p
\end{align}
where $\psi_\infty(x)=e(-x)$ and $\psi_p(x)=e(\mathrm{pr}_p(x))$.  Here $e(x):=e^{2 \pi i x}$ and $\mathrm{pr}_p(x) \in \ZZ[p^{-1}]$ is chosen so that $x-\mathrm{pr}_p(x) \in \ZZ_p$.  

Let $|\cdot|:\A_\QQ^\times \to \RR_{>0}$ denote the adelic norm.  Thus $|\cdot|$ is the product of the usual archimedian norm $|\cdot|_\infty$ and all the $p$-adic norms $|\cdot|_p$. 
Let $\Phi_0=\Phi_{0\infty}\one_{\widehat{\ZZ}^2}$ (regarded as a function on $\A_\QQ^2$), and let $f=f_\infty \one_{V(\widehat{\ZZ})}$.  Here $\widehat{\ZZ}:=\prod_p \ZZ_p$.  Let $$
\Phi(x,y,w):=\Phi_0(x,y)f(w),
$$
regarded as a function on $\A_\QQ^2 \times V(\A_\QQ)$, and set
$$
\Phi^{\mathrm{sw}}(x,y,w):=\Phi(y,x,w).
$$

Let 
\begin{align*}
\langle\,,\,\rangle:V(\QQ) \times V(\QQ) &\lto \QQ\\
(x,y) &\longmapsto x^ty
\end{align*}
be the standard inner product. 
For $\xi \in V(\QQ)$ and (unitary) characters $\chi:\QQ^\times \backslash \A_\QQ^\times \to \CC^\times$ let $\chi_s:=\chi|\cdot|^s$.  In particular $1_s:=|\cdot|^s$.  Let
\begin{align*}
\mathcal{I}(\Phi,\chi_s)(\xi)&=\int_{\A_\QQ^\times \times V(\A_\QQ)}\Phi\left(\frac{Q(w)}{t},t,w \right)\psi\left(\frac{\langle \xi,w \rangle}{t} \right) \chi_s(t)dt^\times dw\\
&=\int_{\RR^\times \times V(\RR)}\Phi_\infty\left(\frac{Q(w)}{t},t,w \right)e\left(\frac{\langle \xi, w\rangle}{t} \right)\chi_{s\,\infty}(t) dt^\times dw \\& \times \prod_{p} \int_{V(\ZZ_p)} \int_{\ZZ_p}\one_{t\widehat{\ZZ}_p}(Q(w))\psi_p\left(\frac{\langle \xi,w \rangle}{t} \right)\chi_{s\,p}(t)dt^\times dw.
\end{align*}
It is not hard to see that this function converges absolutely for $\mathrm{Re}(s)>1$ and vanishes unless $\xi \in N^{-1}V(\ZZ)$ for a sufficiently large integer $N$. To make the classical analogue of this integral clearer, if the determinant of the matrix $J$ associated to the quadratic form $Q$ is in $\GL_n(\ZZ_p)$ for some prime $p \neq 2$, then 
\begin{align*}
\int_{V(\ZZ_p)} \int_{\ZZ_p}\one_{t\widehat{\ZZ}_p}(Q(w))\psi_p\left(\frac{\langle \xi,w \rangle}{t} \right)|t|^s_pdt^\times dw=\one_{V(\ZZ_p)}(\xi) \sum_{k=0}^\infty p^{-kn -ks}\sum_{\substack{w \in V(\ZZ/p^k\ZZ)\\ Q(w) \equiv 0 \pmod{p^k}}}
e\left(\frac{\langle \xi,w\rangle}{p^k}\right).
\end{align*}
When $\xi=0$ this is equal to
\begin{align*}
& \sum_{k=0}^\infty p^{-kn -ks-k}\sum_{\substack{w \in V(\ZZ/p^k\ZZ)\\ Q(w) \equiv 0 \pmod{p^k}}}\sum_{a\in \ZZ/p^k}e\left(\frac{aQ(w)}{p^k}\right)\\
&= \sum_{k=0}^\infty\sum_{i=0}^k p^{-(k-i)n -ks-k}\sum_{\substack{w \in V(\ZZ/p^{k-i}\ZZ)}}\sum_{a \in (\ZZ/p^{k-i})^{\times}}e\left(\frac{aQ(w)}{p^{k-i}}\right)\\
&=\sum_{i=0}^\infty p^{-i(s+1)}\sum_{k=0}^\infty p^{-kn-ks-k}\sum_{\substack{w \in V(\ZZ/p^{k}\ZZ)}}\sum_{a \in (\ZZ/p^k)^{\times}}e\left(\frac{aQ(w)}{p^k}\right).
\end{align*}
This makes the relationship with the usual singular series clear (see below \cite[Lemma 31]{HBnewcircle}, for example). 

Let $Q^{\vee}$ be the quadratic form with matrix $J^{-1}$.  The following is our main theorem over the rationals:
\begin{thm} \label{thm:intro}
 Assume that $n \geq 4$ is even and that $\varepsilon>0$.  If  $(-1)^{n/2}\det J$ is not a square of a rational number
then 
\begin{align*}
N(X)&=X^{n-2}\mathrm{Res}_{s=-1}\left(\mathcal{I}(\Phi,1_s)(0)-\mathcal{I}(\Phi^{\mathrm{sw}},1_s)(0)\right)+
O_\varepsilon(X^{n/2+\varepsilon-1}).
\end{align*}
If $n>4$ and $(-1)^{n/2}\det J$ is the square of a rational number then
\begin{align*}N(X)&=X^{n-2}\mathrm{Res}_{s=-1}\left(\mathcal{I}(\Phi,1_s)(0)-\mathcal{I}(\Phi^{\mathrm{sw}},1_s)(0)\right)
\\&+\sum_{\substack{\xi \in V(\QQ)\\ Q^\vee(\xi)=0}}X^{n/2}\mathrm{Res}_{s=1-n/2}\left(\mathcal{I}(\Phi,1_s)(\xi)-\mathcal{I}(\Phi^{\mathrm{sw}},1_s)(\xi) \right)+O_\varepsilon(X^{n/2+\varepsilon-1}).
\end{align*}

If $n=4$ and $\det J $ is the square of a rational number, then
\begin{align*}
N(X)&=X^2 \log X \lim_{s \to -1}\left(\frac{\mathcal{I}(\Phi,1_s)(0)-\mathcal{I}(\Phi^{\mathrm{sw}},1_s)(0)}{\Lambda(s+2)^2}\right)\\&
+
X^2\mathrm{Res}_{s=-1}\left(\mathcal{I}(\Phi,1_s)(0)-\mathcal{I}(\Phi^{\mathrm{sw}},1_s)(0)\right)\\
&+\sum_{\substack{0 \neq \xi \in V(\QQ)\\ Q^\vee(\xi)=0}}X^{2}\mathrm{Res}_{s=-1}\left(\mathcal{I}(\Phi,1_s)(\xi)-\mathcal{I}(\Phi^{\mathrm{sw}},1_s)(\xi) \right)+O(X^{1+\varepsilon}).
\end{align*}
\end{thm}
\noindent Here $1_s:=|\cdot|^s$.
The generalization of this theorem to number fields is Theorem \ref{thm:main} below.  We note that there are additional terms in the general formula involving integrals $\mathcal{I}(\Phi,\mathcal{G}_s)$ for a certain character $\mathcal{G}$.  The fact that in the introduction we are working over $\QQ$ and letting $\Phi^{\infty}=\one_{\widehat{\ZZ}^2 \times V(\widehat{\ZZ})}$ is responsible for the simplification of the formula above.  The author thanks Heath-Brown for comments related to this fact.    We will explain how to derive the simplified expression above after stating Theorem \ref{thm:main} below.  We also note that under the assumptions of this introduction the functions $\Phi$ and $\Phi^{\mathrm{sw}}$ only differ at the archimedian place, therefore
$$
\mathcal{I}(\Phi,\chi_s)-\mathcal{I}(\Phi^{\mathrm{sw}},\chi_s)=\left(\mathcal{I}_\infty(\Phi,\chi_s)-\mathcal{I}_\infty(\Phi^{\mathrm{sw}},\chi_s)\right)\prod_{p}\mathcal{I}(\one_{\ZZ_p^2 \times V(\ZZ_p)},\chi_s)
$$
This observation makes it easier to see how our expression relates to that of Heath-Brown.  The difference of archimedian functions occurring here is a hallmark of the $\delta$-symbol method.

The assumption that $n$ is even is necessary in order for this theorem to be true as stated.  The point is that the local integrals $\mathcal{I}(\Phi,\chi_s)$ are slightly more complicated when $n$ is odd.  We still think that there is a formula for lower order terms in $N(X)$ when $n$ is odd however, and feel that it is an interesting problem to describe it.

We will outline the proof of the theorem after giving some remarks on related literature.

\subsection{Remarks on related literature}
There are other results in the literature where secondary terms
are obtained via the circle method.  
 We mention a few with no claims to completeness.  First we point out \cite{Lindqvist:weak:approx}, which treats the $F=\QQ$ and $n=4$ case of the paper \cite{HBnewcircle} mentioned above, but with congruence conditions.  Theorem \ref{thm:main} in the special case $F=\QQ$ and $n=4$ recovers her result, with a better error term, but with less explicit analysis of the contribution of ramified places (including $\infty$).   
 Vaughn and Wooley \cite{VaughanWooleyhoexp} and Schindler \cite{SchindlerHigherOrderExpansions}  have investigated secondary and higher order terms in Waring's problem.  However, these terms appear for an entirely different reason, namely that in these works they estimate the number of zeros in a suitable box instead of a smoothed box using a Schwartz function $f$ as above.  Using a Schwartz function eliminates these terms and thus they do not appear in our analysis.  Another case where one can (conditionally) obtain information about a secondary term in the circle method is treated in \cite{HB:circle:diagonal:cubic}.  

The point counting problems we study can also be examined from the point of view of height zeta functions.  It is likely that Theorem \ref{thm:main} implies the meromorphic continuation of a suitable height zeta function to a larger half plane than was previously known.  We refer to \cite{Franke:Manin:Tschinkel} and \cite{ChambertLoir:Lectures} for details on height zeta functions.  

The work in 
\cite{Franke:Manin:Tschinkel} is based on estimating   $N(X)$ via Eisenstein series.  It can also be studied via theta functions or by realizing the zero locus of the quadratic form as a homogeneous space (see \cite{DukeRudnickSarnak}, for example).  From this perspective 
it is probably possible to obtain a secondary term as we have, but the description obtained in this manner is spectral in nature.

In contrast, the description of the secondary term we give is manifestly geometric.  It only involves quadratic forms and Hecke Gr\"o{\ss}encharaktere as opposed to residues of automorphic $L$-functions on nonabelian groups.  The proof and its statement make no use of analytic properties of automorphic representations apart from Hecke Gr\"o{\ss}encharaktere.  This is important because the author hopes that results of this type can be used to prove expected analytic properties of automorphic $L$-functions that are currently unknown.  This is discussed in more detail in \cite{Getz:RSMonoid}, and is intimately related to Langlands' beyond endoscopy idea \cite{LanglandsBE} and Braverman and Kazhdan's ideas on nonabelian Fourier transforms \cite{BK-lifting} (see also work of L.~Lafforgue \cite{LafforgueJJM} and Ng\^o \cite{NgoSums}, as well as the other work cited in \cite{Getz:RSMonoid}).  

We also point out that if one were able to obtain secondary and higher order asymptotics for the number of solutions of higher degree forms (or systems of forms) there would be profound consequences in automorphic representation theory.  The point is roughly that the main term in these problems usually corresponds to the trivial representation when there is an automorphic (spectral) interpretation of the situation.  Terms and estimates that are sufficiently small in magnitude compared to the main term, in contrast, are intimately connected to the cuspidal spectrum, which is really the focus of interest from the point of view of automorphic representation theory.

Some automorphic applications can already be obtained from this work on quadratic forms.  As observed in a special case in \cite{Getz:RSMonoid}, Theorem \ref{thm:intro} can be viewed as giving a summation formula for the subscheme of $V$ defined by the vanishing of the quadratic form $Q$.  Assume for simplicity that $n >4$ and $(-1)^{n/2}\det J  \in (\QQ^\times)^2$ (i.e.~it is a square).  Then Theorem \ref{thm:intro} implies that 
\begin{align*}
&\sum_{\substack{\xi \in V(\QQ)\\Q^\vee(\xi)=0}}\mathrm{Res}_{s=1-n/2}\mathcal{I}(\Phi,1_s)(\xi)\\
&=\lim_{X \to \infty}\frac{N(X)-X^{n-2}\left(\mathcal{I}(\Phi,1_0)(0)-\mathcal{I}(\Phi^{\mathrm{sw}},1_0)(0) \right)}{X^{n/2}}+\sum_{\substack{\xi \in V(\QQ)\\ Q^\vee(\xi)=0}}\mathrm{Res}_{s=1-n/2}\mathcal{I}(\Phi^{\mathrm{sw}},1_s)(\xi).
\end{align*}
The analogous assertion remains true over general number fields but we will not mention it later after stating Theorem \ref{thm:main}.  Interestingly, the limit in this expression can be computed spectrally in some cases; this was the motivation for \cite{Getz:RSMonoid}.  One immediate application of Theorem \ref{thm:main}, the general version of Theorem \ref{thm:intro}, is that the results of loc.~cit. can be generalized from submonoid of $\gl_2 \times \gl_2$ consisting of pairs of matrices with equal determinant to the case where $\gl_2$ is replaced by an arbitrary division algebra and $\det$ is replaced by the reduced norm.  We will not go into the details in this paper.

\subsection{Sketch of the proof} \label{ssec:sketch}

We now discuss the idea of the proof of Theorem \ref{thm:intro}.  The formal proof (for arbitrary number fields) is given in \S \ref{ssec:main} below.
In the notation of \S \ref{ssec:statement} one has
\begin{align} \label{gamma:sum}
N(X):&=\sum_{\substack{\xi \in V(\ZZ)\\ Q(\xi) =0}}f_\infty\left(\frac{\xi}{\Delta(X)} \right)=\sum_{\substack{\xi \in V(\QQ)\\ Q(\xi)=0}}f\left( \frac{\xi}{\Delta(X)}\right)=\sum_{\xi \in V(\QQ)} \delta_{Q(\xi)} f\left(\frac{\xi}{\Delta(X)} \right) .
\end{align}
Here $\Delta(X)$ is viewed as an idele by identifying $\RR^\times$ with its image under the natural embedding  
\begin{align*}
\RR^\times &\lto \RR^\times \times (\A_\QQ^\infty)^\times=\A_\QQ^\times\\
x &\longmapsto (x,1)
\end{align*}
and $\xi$ is viewed as an element of $V(\A_\QQ)$ via the diagonal embedding $V(\QQ) \hookrightarrow V(\A_\QQ)$. Moreover 
\begin{align} \label{deltax}
\delta_x=\begin{cases}1 \textrm{ if } x=0\\
0 \textrm{ otherwise}. \end{cases}
\end{align}

 As in \cite{HBnewcircle}, our main tool in analyzing \eqref{gamma:sum} is an expansion of this $\delta$-symbol due essentially to Duke, Friedlander, and Iwaniec \cite{DFI}.  This has been generalized to number fields by Browning and Vishe \cite{BrowningVisheCircleMethod}.  Unfortunately their generalization is not adelic, and hence is not optimal from the point of view of possible applications to automorphic forms.  In \cite[Proposition 2.1]{Getz:RSMonoid} we gave an adelic expansion of the $\delta$-symbol and it is this expansion that is used in the current paper.  The two expressions are more or less equivalent.   
In any case, applying \cite[Proposition 2.1]{Getz:RSMonoid} 
see that $N(X)$ is equal to 
\begin{align*}
\sum_{\substack{\xi \in V(\QQ)}}\frac{c_{X,\Phi_0}}{X}\sum_{d \in \QQ^\times}\left(\Phi_0\left(\frac{Q(\xi)}{d\Delta(X)},\frac{d}{\Delta(X)} \right)-\Phi_0\left(\frac{d}{\Delta(X)},\frac{Q(\xi)}{d\Delta(X)} \right)\right)f\left(\frac{\xi}{\Delta(X)} \right).
\end{align*}
Here we take $X$ sufficiently large (in a sense depending on $\Phi_0$) and for any $N>0$ one has 
\begin{align} \label{c:func}
c_{X,\Phi_0}=1+O_{\Phi_0,N}(X^{-N}).
\end{align}
Note that since $\Phi_0^\infty=\one_{\widehat{\ZZ}^2}$ the sum on $d$ is really over $d \in \ZZ-0$ that divide $Q(\xi)$.

We now apply Poisson summation in $\xi$ and a  Mellin transform in $d$ to write the above as 
\begin{align*}
\frac{X^{n-1}c_{X,\Phi_0}}{2 \pi i}\sum_{\xi \in V(\QQ)} \sum_{\chi} \int_{\mathrm{Re}(s)=\sigma}X^{s}\left(\mathcal{I}(\Phi,\chi_s)(\xi)-\mathcal{I}(\Phi^{\mathrm{sw}},\chi_s)(\xi)\right)ds.
\end{align*}
Here the sum on $\chi$ is over all characters of $\QQ^\times \backslash \A_{\QQ}^\times$ trivial on $\RR_{>1}$ (viewed as a subgroup of the $\RR^\times$ factor of $\A_\QQ^\times$) and $\sigma \in \RR_{>1}$.  Up to this point everything we have done is the same as in \cite{HBnewcircle}, apart from the expansion of the $\delta$-symbol, which is slightly different here than in loc.~cit.~in that we do not write it in terms of Ramanujan sums.  The integrals $\mathcal{I}(\Phi,\chi_s)(\xi)$ are visibly Euler products.  Heath-Brown computes the finite Euler factors up to degree $p^{2s}$ and bounds the rest of the terms.  This is enough to obtain the asymptotic proven in loc.~cit, and it is roughly equivalent to giving a meromorphic continuation of each $\mathcal{I}(\Phi,\chi_s)(\xi)$ to the range $\mathrm{Re}(s)>1-n/2-\beta$ for some $\beta>0$.  

However, this is not enough to give the asymptotic of Theorem \ref{thm:intro}.  In addition to generalizing everything above to arbitrary number fields (and congruence conditions), our primary contribution in this paper is to realize that the functions $\mathcal{I}(\Phi,\chi_s)(\xi)$ 
admit meromorphic continuations to $\mathrm{Re}(s)>-n/2$.  In fact, we prove that if $S$ is a set of places of $\QQ$ including $\infty$, $2$, and all the places where $\det J$ is not a unit, then
$$
\mathcal{I}^S(\Phi,\chi_s)(\xi),
$$
the factor of $\mathcal{I}(\Phi,\chi_s)(\xi)$ outside of $S$, is holomorphic in the plane if $Q^\vee(\xi) \neq 0$. If $Q^{\vee}(\xi) =0$ it is equal to 
\begin{align*}
\frac{L^S(s+n/2,\chi\mathcal{G})}{L^S(s+1+n/2,\chi \mathcal{G})}\sum_{d|\xi V(\ZZ^S)}\frac{\chi^S(d)}{|d|_S^{s+1}}.
\end{align*}
where $\mathcal{G}$ is the quadratic character of Lemma \ref{lem:global:char}.  This is striking because in the secondary term we again arrive at a sum over zeros, this time of $Q^\vee$.  Moreover, for each fixed $\chi$ and $\xi$ the factor $\mathcal{I}^S(\Phi,\chi_s)(\xi)$  admits a meromorphic continuation to the entire plane.  

As for the places in $S$, the stationary phase method (either over $\RR$ or $\QQ_p$) allows us to control $\mathcal{I}_S(\Phi,\chi_s)(\xi)$ uniformly in $\xi$ and $\chi_s$. Combining this with the observations above and a contour shift we deduce Theorem \ref{thm:intro}.

\subsection{Outline of the paper}

We begin with a section on notation which also reviews the statement of Poisson summation in our context and our normalization of Haar measures.  After this, in \S \ref{sec:asymptotic}, we state and prove our main theorem, the generalization of Theorem \ref{thm:intro} to arbitrary number fields.  In this section we assume the local work of the remaining sections; we feel that this organization makes the overall structure of the argument clearer.  

The local work is completed in \S \ref{sec:arch:bound}-\ref{sec:unram:comp}.  The archimedian computations are in \S \ref{sec:arch:bound}. This is the most technical portion of the paper because we require uniformity of our bounds on the archimedian factors of $\mathcal{I}(\Phi,\chi_s)(\xi)$ both in terms of $\xi$ and the analytic conductor of $\chi_s$.  

The work in the nonarchimedian case is simpler, mostly because the Fourier transform of a compactly supported smooth function is again compactly supported and smooth in this context.  We provide  bounds in the ramified case in \S \ref{sec:nonarch:bound}, and then finish in \S \ref{sec:unram:comp} with the easiest and prettiest arguments in the paper, which deal with the computation of the local factors of $\mathcal{I}(\Phi,\chi_s)(\xi)$ where all the data are unramified.

\subsection{Acknowledgments}
The author thanks V.~Blomer and D.~Schindler
for their comments on the paper \cite{Getz:RSMonoid} from the point of view of the circle method.  Their interest prompted him to write the current paper.  T.~Browning, R.~Heath-Brown, J.~Rouse, Z.~Rudnick, and T.~Wooley made comments on a preliminary draft that significantly improved the exposition.  The author also thanks H.~Hahn for her help with editing and her constant encouragement.

\section{Notation} \label{sec:notation}

\subsection{Quadratic forms}
Let $F$ be a number field with ring of integers $\OO$ and adele ring $\A_F$.  Let $V=\GG_a^n$ for some integer $n$.  
For $F$-algebras $R$ we equip $V(R)$ with the ``standard'' inner product
\begin{align}
\langle, \rangle:V(R) \times V(R) &\lto R\\
(x,y) &\longmapsto x^ty. \nonumber
\end{align}
Let
$\langle\,,\,\rangle_Q :V(F) \times V(F) \to F$ be a 
nondegenerate symmetric bilinear form and let
$$
Q(x)=\tfrac{1}{2}\langle x,x \rangle_{Q}
$$  
be the associated quadratic form.
Let $J \in \GL_n(F)$ be the symmetric matrix such that
\begin{align} \label{J1J2}
\langle x,Jy \rangle:=\langle x,y \rangle_Q.
\end{align}
We write
\begin{align} \label{Qvee}
Q^{\vee}(x):=\langle x,J^{-1}x\rangle=Q(J^{-1}x).
\end{align}

\subsection{Adelic notation}

The usual absolute value on $\A_F$ or $F_v$ will be denoted $|\cdot|$ or $|\cdot|_v$ if the place $v$ is not clear from the context. For complex numbers $c$ we let
$$
|c|_{\mathrm{st}}:=(c\bar{c})^{1/2}.
$$ 
Thus if $v$ is a complex place of $F$ and $c \in F_v \cong \CC$ one has $|c|_{\mathrm{st}}^2=|c|_v$ (the $\mathrm{st}$ is for ``standard'').  

If $S$ is a finite set of places of $F$ we let $\OO^S \subset F$ be the subring of elements integral outside the finite places of $S$. Its profinite completion is
$$
\widehat{\OO}^S:=\prod_{\substack{v \not \in S\\ v \textrm{ finite}}}
\OO_{v}.
$$
We use the notation $F_S:=\prod_{v \in S}F_v$, and $\A_F^S:=\prod_{v \not \in S}'F_v$ (the restricted direct product with respect to the $\OO_{v}$ for $v \not \in S$).  Similar notation involving the subscript and the superscript $S$ will be in use throughout.

\subsection{Poisson summation and Haar measures}

Let $\mathcal{S}(V(\A_F)):=\mathcal{S}(V(F_\infty)) \otimes C_c^\infty(V(\A_F^\infty))$, where $\mathcal{S}(V(F_\infty))$ is the usual Schwartz space.
Let $\psi:F \backslash \A_F \to \CC^\times$ be a nontrivial character.  
We note that for $f \in \mathcal{S}(V(\A_F))$ one has 
\begin{align} \label{PS}
\sum_{\xi \in V(F)} f(\xi)=\sum_{\xi \in V(F)} \widehat{f}(\xi)
\end{align}
where 
$$
\widehat{f}(x)=\int_{V(\A_F)} f(x)\psi(\langle x,y \rangle)dy.
$$
Here and below we always normalize the Haar measure on $V(\A_F)$ so that it is self-dual with respect to the pairing $(x,y) \mapsto \psi\left(\langle x,y\rangle\right)$.  We note that 
$\psi=\prod_{v}\psi_v$ where the product is over all places $v$ of $F$.  We always normalize the local Haar measures on $V(F_v)$ so they are self-dual with respect to the local pairing $(x,y) \mapsto\psi_v(\langle x, y \rangle)$.  Similarly, we normalize the Haar measures on $\A_F$ and $F_v$ so that they are self dual with respect to the pairings $(x,y) \mapsto \psi(xy)$ and $(x,y) \mapsto \psi_v(xy)$, respectively, and finally let $dt^\times$ be the measure on $\A_F^\times$ that is the product of the local measures $\zeta_v(1)\frac{dt}{|t|}$.

\subsection{Local notation}

In \S \ref{sec:arch:bound}, \S \ref{sec:nonarch:bound} and \S \ref{sec:unram:comp} we fix a place $v$ of $F$ (archimedian or nonarchimedian, depending on the section) and let $F:=F_v$, $\OO:=\OO_v$.  We also do this now to explain notation common to these sections.  We let $\mathcal{S}(V(F))$ be the usual Schwartz space of rapidly decreasing functions when $F$ is archimedian and we let $\mathcal{S}(V(F))=C_c^\infty(V(F))$ (compactly supported locally constant functions) when $F$ is nonarchimedian.  
For quasi-characters $\chi:F^\times \to \CC^\times$
and $\Phi \in \mathcal{S}(V(F))$ we define an integral
\begin{align} \label{loc:int}
&\mathcal{I}(\Phi,\chi)(\xi)
:=\int_{F^{\times } \times V(F)}\Phi\left(\frac{Q(w)}{t},t,w\right)\psi\left(\frac{\langle \xi,w \rangle}{t} \right)dw
\chi(t)dt^\times.
\end{align}
  For $w=(w_1,\dots,w_n) \in V(F)$ let
\begin{align}
|w|=\max_{1 \leq i \leq n}\{|w_i|\}.
\end{align}
We also define
\begin{align*}
\Phi^{\mathrm{sw}}(x,y,w):=\Phi(y,x,w).
\end{align*}

\section{The asymptotic formula} \label{sec:asymptotic}

In this section we state and prove our main theorem, Theorem \ref{thm:main}, assuming the local work completed in \S \ref{sec:arch:bound}, \S \ref{sec:nonarch:bound}, and  \S \ref{sec:unram:comp}.  There is no circularity, 
 as these latter sections are independent of the rest of the paper.  Placing the global manipulations first motivates the local arguments that follow.

\subsection{Adelic integrals}
Let
$$
[\GG_{m}]:=A_{\GG_m} F^\times \backslash \A_F^\times
$$
where $A_{\GG_m}$ is $\RR^{\times}_{>0}$ is embedded diagonally into $F_\infty^\times$.  We let $\widehat{[\GG_{m}]}$ denote
the set of characters of $[\GG_{m}]$.  
Let $\chi \in \widehat{[\GG_{m}]}$ and let $s \in \CC$.  
To ease notation set 
$$
\chi_s:=\chi|\cdot|^s,
$$
 where $|\cdot|$ is the idelic norm on $\A_F^\times$.  
For $\Phi \in \mathcal{S}(\A_F^2 \times V(\A_F))$ and $\xi \in V(\A_F)$ we define
\begin{align} \label{adel:int}
\mathcal{I}(\Phi,\chi_s)(\xi):=\int_{\A_F^\times \times V(\A_F)} \Phi\left(\frac{Q(w)}{t},t,w \right)\psi\left(\frac{\langle \xi, w \rangle}{t} \right)dw\chi_s(t)dt^\times.
\end{align}
This is the adelic analogue of the local integral \eqref{loc:int}.  By trivial bounds it converges absolutely for $\mathrm{Re}(s)>1$.  It is Eulerian, so if $\Phi=\Phi_S\Phi^S$ for a finite set of places $S$ it make sense to write
$$
\mathcal{I}(\Phi,\chi_s)=\mathcal{I}_S(\Phi,\chi_s)\mathcal{I}^S(\Phi,\chi_s)
$$  
 where $\mathcal{I}_S$ denotes the factor at $S$ and $\mathcal{I}^S$ denotes the factor away from $S$.

Let $S$ be a finite set of places of $F$ including the infinite places, the dyadic places and the places where $\psi$ is ramified.  We can and do assume that $S$ is large enough that $J \in \GL_n(\OO^S)$ and
$$
\Phi=\Phi_{0S} \one_{\widehat{\OO}^{S2}} \otimes \Phi_{1S} \one_{V(\widehat{\OO}^S)}
$$
where $\Phi_{1S} \in \mathcal{S}(V(F_S))$ and $\Phi_{0S} \in \mathcal{S}(F_S^{2})$.

For a place $v$ of $F$ and $t \in F_v^\times$ let
$$
\gamma_v(t^{-1}Q)=\gamma(\psi_v(t^{-1}Q))
$$
be the number attached by Weil to the map 
\begin{align*}
V(F_v) &\lto \CC^\times\\
w &\longmapsto \psi_v(t^{-1}Q(w)),
\end{align*}
 viewed as a character of second degree \cite{Weil:Certains:groupes}.  This is the sign of a certain Gauss sum attached to $Q$.

\begin{lem} \label{lem:global:char}
There is a unique character $\mathcal{G} \in \widehat{[\GG_m]}$ such that for every place $v \not \in S$ and every $t \in F_v^\times$ one has
$$
\mathcal{G}_v(t)=\gamma_v(t^{-1}Q).
$$
\end{lem}

\begin{proof}
This is clear from Lemma \ref{lem:nonarch:Gaussian}.
\end{proof}

For a vector space $W$ and $w \in W$ let $\delta_w$ be the $\delta$-symbol defined as in \eqref{deltax}.  If $\chi$ is a character let 
$$
\delta_{\chi}:=\begin{cases} 1 &\textrm{ if } \chi \textrm{ is trivial, and }\\
0 & \textrm{ otherwise.}\end{cases}
$$
For quasi-characters $\chi:F^\times \backslash \A_F^\times \to \CC^\times$ let
$$
\mathcal{C}(\chi):=|\OO/\mathfrak{f}_{\chi}|\prod_{v|\infty}\mathcal{C}(\chi_v)
$$
be the analytic conductor of $\chi$, where $\mathfrak{f}_\chi \subset \OO$ is the usual conductor of $\chi$ and $\mathcal{C}(\chi_v)$ is the archimedian conductor at the infinite place $v$ defined as in \eqref{an:cond}.

The following is the main theorem of this subsection, which amounts to collecting the work of \S \ref{sec:arch:bound}, \S \ref{sec:nonarch:bound} and \S \ref{sec:unram:comp}: 
\begin{thm} \label{thm:global:bound}Let $\chi \in \widehat{[\GG_m]}$ and let $s \in \CC$. 
For $\xi \in V(F)$ the integral $\mathcal{I}(\Phi,\chi_s)(\xi)$ admits a meromorphic continuation to $\mathrm{Re}(s)>-n/2$.  For any $N >0$ and $\sigma_1>\sigma_2>-n/2$ one has
\begin{align*}
\frac{\mathcal{I}(\Phi,\chi_s)(\xi)}{s^{\delta_{\xi}\delta_{\chi}}(s+1)^{\delta_{\xi}\delta_{\chi}}(s+n/2-1)^{\delta_{Q(\xi)}\delta_{\mathcal{G}\chi}}} \ll_{\sigma_1,\sigma_2,N}\mathcal{C}(\chi_s)^{-N}\max(|\xi|,1)^{-N}
\end{align*}
for $\sigma_1>\mathrm{Re}(s)>\sigma_2$.
There is an element $\beta \in \OO \cap F^\times$ and an ideal $\mathfrak{f} \subset \OO$ such that 
$\mathcal{I}(\Phi,\chi_s)(\xi)$ vanishes unless
  $\xi \in \beta^{-1}V(\OO)$ and the conductor of $\chi$ divides $\mathfrak{f}$.
\end{thm}

\begin{proof}
Assume first that $\xi \neq 0$.  Let $A>0$ be as in Theorem \ref{thm:gammaneq0} and let $N>0$.
By Theorem \ref{thm:gammaneq0}, Theorem \ref{thm:nonarch:bound} and Theorem \ref{thm:nonarch:comp} we see that
\begin{align*}
&\frac{\mathcal{I}(\Phi,\chi_s)(\xi)}{L(s+n/2,\chi\mathcal{G}^\infty)^{\delta_{Q(\xi)}}}
\\
&\ll_{\Phi,\sigma_1,\sigma_2,N}\mathcal{C}(\chi_{s})^{-N}\left(\prod_{v|\infty}\max(|\xi|_v,1)^{-N}\min(|\xi|_v,1)^{-A}\right)\left(\prod_{v\nmid \infty}\max (1,|\xi|_v^{1-n/2})\right).\end{align*}
Moreover, $\mathcal{I}(\Phi,\chi_s)(\xi)$ vanishes unless $ \xi \in \beta^{-1}V(\OO)$ and the conductor of $\chi$ divides $\mathfrak{f}$ for some $\beta \in \OO \cap F^\times$ and ideal $\mathfrak{f} \subseteq \OO$ by Lemma \ref{lem:nonarch:support}.  For the remainder of the $\xi \neq 0$ case of the proof we assume the conductor of $\chi$ divides $\mathfrak{f}$.

Using \cite[Lemma 3.4]{Getz:RSMonoid} to handle the factor involving $A$ and absorbing the factor corresponding to $v \nmid \infty$ we deduce that for any $N>0$ one has 
\begin{align*}
&\frac{\mathcal{I}(\Phi,\chi_s)(\xi)}{L^\infty(s+n/2,\chi\mathcal{G})}
\ll_{\Phi,\sigma_1,\sigma_2,N}\mathcal{C}(\chi_{ s})^{-N}\prod_{v|\infty}\max(|\xi|_v,1)^{-N}.
\end{align*}
Recall that for sufficiently large $A>0$
and $-\frac{1}{2} \leq \mathrm{Re}(s) \leq \frac{3}{2}$
 one has a preconvex bound
$$
(s-1)^{\delta_{\chi\mathcal{G}}}L^\infty(s,\chi\mathcal{G}) \ll_{\mathfrak{f}} \mathcal{C}(\chi\mathcal{G}_{s})^A \ll_Q \mathcal{C}(\chi_{ s})^A
$$
for any character $\chi \in \widehat{[\GG_m]}$ of conductor dividing $\mathfrak{f}$
\cite[\S III.6, Theorem 14A]{Moreno:Advanced:analytic:number:theory}.   Thus we deduce the theorem provided that $\xi \neq 0$.

Assume now that $\xi =0$.  By our choice of $S$ and Theorem \ref{thm:nonarch:comp} one has
\begin{align*}
\mathcal{I}(\Phi,\chi_s)(0)=\frac{\mathcal{I}_S(\Phi,\chi_s)(0)L^S(s+1,\chi)L^S(s+n/2,\chi\mathcal{G})}{L^S(s+n/2+1,\chi\mathcal{G})}.
\end{align*}
The factor $L^S(s+n/2+1,\chi\mathcal{G})$ converges absolutely in plane $\mathrm{Re}(s)>-n/2$ and is bounded independently of $\chi$ in that region, so for the purposes of this proof we can ignore this factor.  We also note that by Lemma \ref{lem:nonarch:support} there is an ideal $\mathfrak{f} \subseteq \OO$ depending on $\Phi$ such that  $\mathcal{I}(\Phi,\chi_s)(0)$ vanishes unless the conductor of $\chi$ divides $\mathfrak{f}$.  We assume for the remainder of the proof that the conductor of $\chi$ divides $\mathfrak{f}$.  

For $\frac{3}{2} \geq \mathrm{Re}(s)>\tfrac{1}{4}$ and sufficiently large $A>0$ one has a preconvex bound 
\begin{align} \label{preconvex}
(s-1)^{\delta_{\chi}}L^S(s,\chi) \ll \mathcal{C}(\chi_{ s})^{A}
\end{align}
\cite[\S III.6, Theorem 14A]{Moreno:Advanced:analytic:number:theory}.  Combining this with Theorem \ref{thm:gammaeq0} to handle the $\mathcal{I}_{\infty}(\Phi,\chi_s)$ factor and Theorem \ref{thm:nonarch:bound} to handle the $\mathcal{I}_{S \setminus \infty}(\Phi,\chi_s)$ factor we deduce the theorem in the range $\mathrm{Re}(s)>-\frac{3}{4}$.

We now use the functional equation of $L^S(s,\chi)$ to proceed.  Let
$$
\gamma(s,\chi_S,\psi_S):=\prod_{v \in S}\gamma(s,\chi_v,\psi_v)
$$
where 
$$
\gamma(s,\chi_v,\psi_v)=\frac{L(1-s,\bar{\chi}_v)\varepsilon(s,\chi_v,\psi_v)}{L(s,\chi_v)}
$$
is the usual $\gamma$-factor.  

One then has the functional equation 
\begin{align} \label{functional:eq}
L^S(s,\chi)=\gamma(s,\chi_S,\psi_S)L^S(1-s,\overline{\chi}).
\end{align}
Assume for the remainder of the proof that $-\frac{1}{2} \geq \mathrm{Re}(s)> \sigma_2$.  

We then use \eqref{functional:eq} to write
\begin{align} \label{smaller:s}
\mathcal{I}(\Phi,\chi_s)(0)=\frac{\gamma(s+1,\chi_S,\psi_S)\mathcal{I}_S(\Phi,\chi_s)(0)L^S(-s,\bar{\chi})L^S(s+n/2,\chi)}{L^S(s+n/2+1,\chi)}.
\end{align}
For $-\frac{1}{2} \geq \mathrm{Re}(s)>\sigma_2$ the denominator here is bounded independently of $\chi$ and one has a bound of 
$$
\gamma(s+1,\chi_\infty,\psi_\infty)\mathcal{I}_\infty(\Phi,\chi_{s})(0)\ll_{\Phi,N,\sigma_2}\mathcal{C}(\chi_{ s})^{-N}
$$
by Theorem \ref{thm:gammaeq0}.  The corresponding factor at the places in $S\setminus \infty$ is $\gamma(s+1,\chi_S^\infty,\psi_S^\infty)\mathcal{I}_{S \setminus \infty}(\Phi,\chi_{s})(0)$, which is bounded by a constant depending on $\mathfrak{f}$ and $\sigma_2$ by trivial bounds on the $\gamma$-factor and Theorem \ref{thm:nonarch:bound}.  
Finally, for sufficiently large $A>0$ one has
\begin{align*}
(s+n/2-1)^{\delta_{\chi\mathcal{G}}}L^S(s+n/2,\chi\mathcal{G}) \ll_{\mathfrak{f},\sigma_2} \mathcal{C}(\chi_{s})^{A}\\
(s+1)^{\delta_{\chi}}L^S(-s,\bar{\chi})\ll_{\mathfrak{f},\sigma_2}\mathcal{C}(\chi_{ s})^{A}
\end{align*}
by the preconvex bound \eqref{preconvex} together with trivial bounds at the places in $S \setminus \infty$.  Combining these bounds we deduce the requisite bound on \eqref{smaller:s} (in the range $-\frac{1}{2}  \geq \mathrm{Re}(s)>\sigma_2$). 
\end{proof}

\subsection{The main theorem}
 \label{ssec:main}

Let 
\begin{align}
\Delta:\RR_{>0} &\lto F_{\infty}^\times
\end{align}
be the map that sends $X$ to $X^{[F:\QQ]^{-1}}$ embedded diagonally.  Embedding $F_{\infty}^\times$ diagonally into $V(F_\infty)$ we obtain
\begin{align}
\Delta:\RR_{>0} \lto V(F_\infty).
\end{align}

Let $f \in \mathcal{S}(V(\A_F))$.
Our goal is to estimate
\begin{align} \label{to:estimate}
N(X)=\sum_{\substack{\xi \in V(F): Q(\xi)=0}} f\left(\frac{\xi}{\Delta(X)} \right).
\end{align}
Let $\Gamma_{F\infty}(s):=\prod_{v|\infty}\Gamma_{F_v}(s)$, where
\begin{align} \label{Gamma:def}
\Gamma_{F_v}(s)=\begin{cases} \pi^{-s/2}\Gamma(s/2)& \textrm{ if }F_v \textrm{ is real, and} \\ 2(2\pi)^{-s}\Gamma(s) & \textrm{ if }F_v \textrm{ is complex.} \end{cases}
\end{align}

Let $\Phi_0 \in \mathcal{S}(\A_F^2)$ be such that $\Phi_0(t,0)=0$ for all $t \in \A_F$ and $\mathcal{F}_2(\Phi_0)(0,0) =\Gamma_{F\infty}(1)$, where
$$
\mathcal{F}_2(\Phi_0)(x,y):=\int_{\A_F} \Phi_0(x,t)\psi(yt)dt
$$
is the Fourier transform in the second variable.
For $\Phi \in \mathcal{S}(\A_F^2 \times V(\A_F))$ we define
$$
\Phi^{\mathrm{sw}}(x,y,w)=\Phi(y,x,w).
$$
Let 
\begin{align}
\ell_{F}:=\mathrm{Res}_{s=1}\Lambda_F(s)
\end{align}
where $\Lambda_F(s)=\Gamma_F(s)\zeta_F(s)$ is the completed Dedekind $\zeta$-function of $F$.
We are now in a position to state the main theorem of this paper, the extension of Theorem \ref{thm:intro} to arbitrary number fields:
\begin{thm} \label{thm:main} Let $\varepsilon>0$.  Assume that $n$ is even.  If $n>4$ or $n=4$ and  $(-1)^{n/2}\det J =\det J \not \in (F^\times)^2$
then 
\begin{align*}
N(X)&=\frac{X^{n-2}}{\ell_F}\mathrm{Res}_{s=-1}\left(\mathcal{I}(\Phi,1_s)(0)-\mathcal{I}(\Phi^{\mathrm{sw}},1_s)(0)\right)
\\&+\sum_{\substack{\xi \in V(F)\\ Q^\vee(\xi)=0}}\frac{X^{n/2}}{\ell_F}\mathrm{Res}_{s=1-n/2}\left(\mathcal{I}(\Phi,\mathcal{G}_s)(\xi)-\mathcal{I}(\Phi^{\mathrm{sw}},\mathcal{G}_s)(\xi) \right)+O_\varepsilon(X^{n/2+\varepsilon-1})
\end{align*}
If $n=4$ and $\det J \in (F^\times)^2$ then
\begin{align*}
N(X)&=X^2 \log X \ell_F \lim_{s \to -1}\left(\frac{\mathcal{I}(\Phi,1_s)(0)-\mathcal{I}(\Phi^{\mathrm{sw}},1_s)(0)}{\Lambda_F(s+2)^2}\right)\\&+
\frac{X^2}{\ell_F}\mathrm{Res}_{s=-1}\left(\mathcal{I}(\Phi,1_s)(0)-\mathcal{I}(\Phi^{\mathrm{sw}},1_s)(0)\right)\\
&+\sum_{\substack{\xi \in V(F) - 0\\ Q^\vee(\xi)=0}}\frac{X^{2}}{\ell_F}\mathrm{Res}_{s=-1}\left(\mathcal{I}(\Phi,\mathcal{G}_s)(\xi)-\mathcal{I}(\Phi^{\mathrm{sw}},\mathcal{G}_s)(\xi) \right)+O_\varepsilon(X^{1+\varepsilon}).
\end{align*}

\end{thm}

Before embarking on the proof we stop to emphasize that the terms occurring in this expression are actually quite simple, at least outside of the finite set of places $S$.  Assume that $Q^\vee(\xi)=0$.
Then for any $\chi \in \widehat{[\GG_m]}$ the integral $\mathcal{I}^S(\Phi,\chi_s)$ vanishes if $\chi$ is ramified outside of $S$.  If it is unramified outside of $S$ then
$$
\mathcal{I}^S(\Phi,\chi_s)(\xi)=\frac{L^S(s+n/2,\chi\mathcal{G})}{L^S(s+n/2+1,\chi\mathcal{G})}\sum_{d}\one_{dV(\widehat{\OO}^S)}(\xi)\chi_{s+1}(d)
$$
where the sum on $d$ is over the ideals of $\OO^S$.  Here we identify $d$ with a generator (in $\A_F^{S \times}$) of the ideal $d\widehat{\OO}^S$ (see Theorem \ref{thm:nonarch:comp}).  If $\xi=0$ the sum over $d$ becomes $L^S(s+1,\chi)$.

Thus if  $\xi \neq 0$ (for example) then
$$
\frac{1}{\ell_F}\mathrm{Res}_{s=1-n/2}\mathcal{I}(\Phi,\mathcal{G}_s)(\xi)=\frac{\mathcal{I}_S(\Phi,\mathcal{G}_s)(\xi)}{\Lambda_{FS}(1)L^S(2,\mathcal{G})}\sum_{d}\one_{dV(\widehat{\OO}^S)}(\xi)(|d|^{S})^{2-n/2}.
$$
The situation when $\xi=0$, $\mathcal{G}$ is trivial, and $n=4$ is more complicated, but it is still easy to compute the factors outside of $S$ in this case using \eqref{smaller:s}.

As promised after the statement of Theorem \ref{thm:intro}, we now explain why the expression above simplifies in the setting of the introduction.  Recall that in the introduction we restricted to the case $F=\QQ$ and $\Phi^\infty=\one_{\widehat{\ZZ}^2 \times V(\widehat{\ZZ})}$.  
The function
\begin{align*}
\mathcal{I}^\infty(\one_{\widehat{\ZZ}^2 \times V(\widehat{\ZZ})},\chi_s)
\end{align*}
vanishes unless $\chi$ is unramified at every finite place, and $A_{\GG_m} \QQ^\times \backslash \A_\QQ^\times/\widehat{\ZZ}^\times=1$.  Thus $\mathcal{I}(\one_{\widehat{\ZZ}^2 \times V(\widehat{\ZZ})},\mathcal{G}_s)=0$ when $\mathcal{G}$ is nontrivial, which is to say $(-1)^{n/2}\det J$ is not a square in $\QQ^\times$.   However, even over $\QQ$, if we take $\Phi^\infty$ to be a function that picks out $\xi$ in a particular residue class modulo $V(\ZZ)$, there is no reason to expect that the extra terms in Theorem \ref{thm:main} are nonzero.  This phenomenon (when $F=\QQ$ and $n=4$) is the subject of the paper \cite{Lindqvist:weak:approx}.

Before proving the theorem we prove a lemma:
\begin{lem} \label{lem:iszero}
One has
\begin{align*}
\mathrm{Res}_{s=0}X^s\left(\mathcal{I}(\Phi,1_s)(0)-\mathcal{I}(\Phi^{\mathrm{sw}},1_s)(0) \right)=0.
\end{align*}
\end{lem}

\begin{proof} 
By Theorem \ref{thm:nonarch:comp} one has 
$$
\mathcal{I}(\one_{(\widehat{\OO}^S)^2 \times V(\widehat{\OO}^S)},1_s)(0)=\frac{L^S(s+n/2,\mathcal{G})\zeta^S_F(s+1)}{L^S(s+n/2+1,\mathcal{G})}.
$$
Thus it suffices to verify that 
\begin{align} \label{equal:0}
\mathcal{I}(\Phi_S,1_0)(0)=\mathcal{I}(\Phi_S^{\mathrm{sw}},1_0)(0).
\end{align}
Taking changes of variables $t \mapsto Q(v)t$ and then $t \mapsto t^{-1}$ we see that 
\begin{align*}
\mathcal{I}_S(\Phi^{\mathrm{sw}},1_0)(0)
&=\int_{F_S^\times \times V(F_S)}\Phi_S\left(t,\frac{Q(v)}{t},v\right)dvdt^\times\\
&=\int_{F_S^\times \times V(F_S)}\Phi_S(Q(v)t,t^{-1},v)dv dt^\times\\
&=\int_{F_S^\times \times V(F_S)}\Phi_S\left(\frac{Q(v)}{t},t,v\right)dvdt^\times\\
&=\mathcal{I}_S(\Phi,1_0)(0).
\end{align*}
Thus \eqref{equal:0} is valid.  Here to justify the changes of variables we require the absolute convergence of the integral over $F_S^\times \times V(F_S)$.  To check this we note that $|\Phi_S|$ can be bounded by a nonnegative Schwartz function so we can assume that $\Phi_S$ (and hence $\Phi^{\mathrm{sw}}_S$) is nonnegative.  Under this assumption, by the Fubini-Tonelli theorem it is enough to check that the integral over $t$ converges absolutely.  At the archimedian places in $S$ this is explicitly stated in Theorem \ref{thm:gammaeq0}.  At the nonarchimedian places it is an easy consequence of the arguments proving Theorem \ref{thm:nonarch:bound}. \end{proof}

\begin{proof}[Proof of Theorem \ref{thm:main}]
Using the expansion of the $\delta$-symbol given in \cite[Proposition 2.1]{Getz:RSMonoid}  we write 
\begin{align*}
N(X)&=\sum_{\substack{ \xi \in V(F)}}\delta_{Q(\xi)}f\left(\frac{\xi}{\Delta(X)}\right)\\
&=\sum_{\substack{\xi \in V(F)}}\frac{c_{X,\Phi_0}}{X}\sum_{d \in F^\times}\left(\Phi_0\left(\frac{Q(\xi)}{d\Delta(X)},\frac{d}{\Delta(X)} \right)-\Phi_0\left(\frac{d}{\Delta(X)},\frac{Q(\xi)}{d\Delta(X)} \right)\right)f\left(\frac{\xi}{\Delta(X)} \right).
\end{align*}
Here we take $X$ sufficiently large (in a sense depending on $\Phi_0$) and for any $N>0$ one has 
\begin{align} \label{c:func}
c_{X,\Phi_0}=\frac{1}{\Gamma_{F\infty}(1)}+O_{\Phi_0,N}(X^{-N}).
\end{align}

It is not hard to see that the double sum over $\xi$ and $d$ here is absolutely convergent.  In particular we can exchange the sum over $d$ and the sum over $\xi$.  For each $d$ the function
$$\left(\Phi_0\left(\frac{Q(\xi)}{d\Delta(X)},\frac{d}{\Delta(X)} \right)-\Phi_0\left(\frac{d}{\Delta(X)},\frac{Q(\xi)}{d\Delta(X)} \right)\right)f\left(\frac{\xi}{\Delta(X)} \right)
$$
is Schwartz as a function of $\xi \in V(\A_F)$.  We can therefore apply Poisson summation in $\xi \in V(F)$ (see \eqref{PS}) to arrive at
\begin{align*}
&\frac{c_{X,\Phi_0}}{X}\sum_{d \in F^\times}\sum_{\substack{\xi \in V(F)}}\int_{V(\A_F)}\left(\Phi_0\left(\frac{Q(w)}{d\Delta(X)},\frac{d}{\Delta(X)} \right)-\Phi_0\left(\frac{d}{\Delta(X)},\frac{Q(w)}{d\Delta(X)} \right)\right) f\left(\frac{w}{\Delta(X)} \right) \psi\left( \frac{\langle \xi,w\rangle}{d} \right)dw.
\end{align*}
Here we have changed variables $\xi \mapsto d^{-1}\xi$ in $\xi$.  
Again it is not difficult to see that the resulting double sum over $\xi$ and $d$ is absolutely convergent, so we bring the sum over $\xi$ outside the sum over $d$ and apply Poisson summation in $d \in F^\times$ to arrive at
\begin{align*}
\sum_{\xi \in V(F)}\frac{c_{X,\Phi_0}}{2\pi iX z_F}\sum_{\chi \in \widehat{[\GG_{m}]}}\int_{\mathrm{Re}(s)=\sigma}&\int_{V(\A_F) \times \A_F^{ \times}}\left(\Phi_0\left(\frac{Q(w)}{t\Delta(X)},\frac{t}{\Delta(X)} \right)-\Phi_0\left(\frac{t}{\Delta(X)},\frac{Q(w)}{t\Delta(X)} \right)\right)\\
&\times f\left(\frac{w}{\Delta(X)} \right)\psi\left(\frac{\langle  \xi ,w \rangle}{t} \right)dw
\chi_s(t)dt^\times  ds.
\end{align*}
where $z_F:=\mathrm{Res}_{s=1}\zeta_F(s)$.
A reference for this application of Poisson summation is \cite[\S 2]{Blomer_Brumley_Ramanujan_Annals}.  We still have to check that the application is justified, but we take this up in a moment.

We change variables $(t,w) \mapsto (\Delta(X)t,\Delta(X)w)$ to arrive at
\begin{align} \label{after:PS}
\nonumber & \sum_{\xi \in V(F)}\sum_{\chi \in \widehat{[\GG_{m}]}}\int_{\mathrm{Re}(s)=\sigma}\frac{c_{X,\Phi_0} X^{s}}{2\pi iX^{1-n}z_F}\int_{V(\A_F) \times \A_F^{ \times }}\left(\Phi_0\left(\frac{Q(w)}{t},t \right)-\Phi_0\left(t,\frac{Q(w)}{t} \right)\right)\\& \times f\left(w \right)\psi\left(\frac{\langle \xi ,w \rangle}{t} \right)dw \chi_s(t)dt^\times ds \nonumber\\
&=\frac{c_{X,\Phi_0} X^{n-1}}{2\pi iz_F} \sum_{\xi \in V(F)} \sum_{\chi \in \widehat{[\GG_m]}}\int_{\mathrm{Re}(s)=\sigma}X^{s}\left(\mathcal{I}(\Phi,\chi_s)(\xi)-\mathcal{I}(\Phi^{\mathrm{sw}},\chi_s)(\xi) \right) ds. 
\end{align}
To check that the application of Poisson summation in $d$ is justified it suffices to check that 
\begin{align*}
\sum_{\xi \in V(F)} \sum_{\chi \in \widehat{[\GG_m]}}\int_{\mathrm{Re}(s)=\sigma}\left(\left|\mathcal{I}(\Phi,\chi_s)(\xi)\right|+\left|\mathcal{I}(\Phi^{\mathrm{sw}},\chi_s)(\xi) \right|\right) ds<\infty
\end{align*}
for $\sigma>0$.  But this is an easy consequence of Theorem \ref{thm:global:bound}.

We now take $\varepsilon>0$ and shift the contour to $\mathrm{Re}(s)=\varepsilon-n/2$.  In view of Theorem \ref{thm:global:bound} we then see that \eqref{after:PS} is equal to 
\begin{align} \label{Resat0:term}
&
\frac{c_{X,\Phi_0} X^{n-1}}{z_F}\mathrm{Res}_{s=0}\left(X^s\left( \mathcal{I}(\Phi,1_s)(0)-\mathcal{I}(\Phi^{\mathrm{sw}},1_s)(0)\right) \right)\\\label{0:term}
&+\frac{c_{X,\Phi_0} X^{n-1}}{z_F}\mathrm{Res}_{s=-1}\left(X^s\left( \mathcal{I}(\Phi,1_s)(0)-\mathcal{I}(\Phi^{\mathrm{sw}},1_s)(0)\right) \right)\\ \label{minus:term}
&-\frac{c_{X,\Phi_0} X^{n-1}}{z_F}\mathrm{Res}_{s=-1}\left(X^s\left( \mathcal{I}(\Phi,\mathcal{G}_s)(0)-\mathcal{I}(\Phi^{\mathrm{sw}},\mathcal{G}_s)(0)\right) \right)^{\delta_{4-n}\delta_{\mathcal{G}}}
\\ \label{smaller:res}
&+\frac{c_{X,\Phi_0} X^{n-1}}{z_F}\sum_{\substack{\xi \in V(F)\\ Q(\xi)=0 }}\mathrm{Res}_{s=1-n/2}\left(X^s\left( \mathcal{I}(\Phi,\mathcal{G}_s)(\xi)-\mathcal{I}(\Phi^{\mathrm{sw}},\mathcal{G}_s)(\xi)\right) \right)
\\&+
\frac{c_{X,\Phi_0} X^{n-1}}{2\pi iz_F} \sum_{\xi \in V(F)} \sum_{\chi \in \widehat{[\GG_m]}}\int_{\mathrm{Re}(s)=\varepsilon-n/2}X^{s}\left(\mathcal{I}(\Phi,\chi_s)(\xi)-\mathcal{I}(\Phi^{\mathrm{sw}},\chi_s)(\xi) \right) ds. \label{remainder}
\end{align}
The term with a negative sign \eqref{minus:term} occurs because the sum of \eqref{0:term} and \eqref{smaller:res} overcounts the residues when $\mathcal{G}$ is trivial and $n=4$.  The term \eqref{Resat0:term} is zero by Lemma \ref{lem:iszero}.  By Theorem \ref{thm:global:bound} and \eqref{c:func} the term \eqref{remainder} is $O_{\Phi,\varepsilon}(X^{\varepsilon+n/2-1})$.  

By Theorem \ref{thm:global:bound} if $\xi \neq 0$ then the pole of  $\mathcal{I}(\Phi,\mathcal{G}_s)(\xi)$ at $s=1-n/2$ is simple.  Combining this fact with the asymptotic for $c_{X,\Phi_0}$ from \eqref{c:func} we see that \eqref{smaller:res} is equal to $O_{\Phi_0,N}(X^{-N})$ plus
\begin{align*}
 \frac{ X^{n/2}}{\ell_F}\sum_{\substack{\xi \in V(F)\\ Q^\vee(\xi)=0 }}\mathrm{Res}_{s=1-n/2}\left(\left( \mathcal{I}(\Phi,\mathcal{G}_s)(\xi)-\mathcal{I}(\Phi^{\mathrm{sw}},\mathcal{G}_s)(\xi)\right) \right)
\end{align*}
for any $N>0$.  Assume for the moment that 
$n>4$ or $(-1)^{n/2}\det J \not \in (F^{\times})^2$.  Then by Theorem \ref{thm:global:bound} $\mathcal{I}(\Phi,\chi_s)(0)$ has only simple poles for $\mathrm{Re}(s)>-n/2$.  Thus when $n>4$ or $(-1)^{n/2}\det J \not \in (F^{\times})^2$ the terms \eqref{Resat0:term}, \eqref{0:term}, and \eqref{minus:term} can be evaluated in exactly the same manner as we evaluated \eqref{smaller:res}, proving the theorem in this case.

We now handle the case where $(-1)^{n/2}\det J \in (F^\times)^2$ and $n=4$.  In this case $\mathcal{G}$ is trivial, and hence  the sum of \eqref{Resat0:term}, \eqref{0:term}, and \eqref{minus:term} is
\begin{align*}
\frac{ X^{3}}{\ell_F}\mathrm{Res}_{s=-1}\left(X^s\left( \mathcal{I}(\Phi,1_s)(0)-\mathcal{I}(\Phi^{\mathrm{sw}},1_s)(0)\right) \right)+O_{\Phi_0,N}(X^{-N})
\end{align*}
for any $N>0$.  Here we have used \eqref{c:func} to replace $c_{X,\Phi_0}$ by its asymptotic value $\Gamma_{F\infty}(1)^{-1}$.  Since $n=4$, the pole of $\mathcal{I}(\Phi,1_s)(0)$ at $s=-1$ is of order $2$.  Thus the above equal to 
\begin{align*}
\frac{ X^{2}}{\ell_F}\Bigg( &\log X\lim_{s \to -1}\left((s+1)^2
\left( \mathcal{I}(\Phi,1_s)(0)-\mathcal{I}(\Phi^{\mathrm{sw}},1_s)(0)\right)\right)\\&+ \mathrm{Res}_{s=-1}\left(\mathcal{I}(\Phi,1_s)(0)-\mathcal{I}(\Phi,1_s)(0)\right)\Bigg)+O_{\Phi_0,N}(X^{-N}).
\end{align*}
To complete the proof we now note that
\begin{align*}
\lim_{s \to -1}\left((s+1)^2\mathcal{I}(\Phi,1_s)(0)\right)&=\lim_{s \to -1}\left((s+1)^2\Lambda_F(s+2)^2 \frac{\mathcal{I}(\Phi,1_s)(0)}{\Lambda_F(s+2)^2}\right)\\
&=\ell_F^2 \lim_{s \to -1}\frac{\mathcal{I}(\Phi,1_s)(0)}{\Lambda_F(s+2)^2}
\end{align*}
and similarly for $\Phi$ replaced by $\Phi^{\mathrm{sw}}$.
\end{proof}

\section{Archimedian bounds} \label{sec:arch:bound}

\subsection{Main result} For this section we fix an archimedian place $v$ of $F$ and omit it from notation, writing $F:=F_v$.  For each quasi-character $\chi:F^\times \to \CC$
there is a complex number $\mu_\chi$ such that 
$$
L(s,\chi)=\Gamma_F(s+\mu_\chi)
$$
with $\Gamma_F$ defined as in \eqref{Gamma:def}.
The analytic conductor of $\chi$ is then 
\begin{align} \label{an:cond}
C(\chi):=1+|\mu_{\chi}|_{\mathrm{st}}.
\end{align}
 The main result of this section is the following theorem:

\begin{thm}\label{thm:gammaneq0} Let $\sigma_1>\sigma_2>-n/2$ and $N \in \ZZ_{>0}$. Assume $\xi \neq 0$.  For $\sigma_1>\mathrm{Re}(s)>\sigma_2$ the integral over $t$ in the definition of $\mathcal{I}(\Phi,\chi_s)(\xi)$ converges absolutely.  There is an $A>0$ such that for $\sigma_1>\mathrm{Re}(s)>\sigma_2$ one has 
\begin{align*}
\mathcal{I}(\Phi,\chi_s)(\xi) \ll_{N,\sigma_1,\sigma_2} \max(|\xi|,1)^{-N}\min(|\xi|,1)^{-A}\mathcal{C}(\chi_{s})^{-N}.
\end{align*}
\end{thm}

We also require the following companion statement when $\xi=0$:
\begin{thm}  \label{thm:gammaeq0}
Let $N \in \ZZ_{>0}$, 
$
\sigma_1>-\tfrac{3}{4}>\sigma_2>-n/2.
$  
For $s \in \CC$ with $\sigma_1>\mathrm{Re}(s)>-\frac{3}{4}$ the integral over $t$ in the definition of $\mathcal{I}(\Phi,\chi_s)(0)$ converges absolutely and one has
\begin{align*}
\mathcal{I}(\Phi,\chi_s)(0) \ll_{N,\sigma_1} \mathcal{C}(\chi_s)^{-N}.
\end{align*}
The function $\mathcal{I}(\Phi,\chi_s)(0)$ can be meromorphically continued to $\mathrm{Re}(s)>-n/2$ and for  
$
-\frac{1}{4} \geq \mathrm{Re}(s)>\sigma_2
$ one has an estimate
\begin{align}
\gamma(\chi_{s+1},\psi)\mathcal{I}(\Phi,\chi_s)(0) \ll_{N,\sigma_2} \mathcal{C}(\chi_s)^{-N}
\end{align}
where $\gamma(\chi_s,\psi)$ is the usual $\gamma$-factor (see \eqref{gamma:factor}).
\end{thm}
The argument is essentially the same as that proving \cite[Theorem 5.1 and Proposition 5.4]{Getz:RSMonoid}, but the current setting 
is a broad generalization of that of loc.~cit., so we give a complete proof.

We observe that it suffices to prove the theorem under the following simplifying assumption:
\begin{itemize}
\item[(A)] One has $\psi(x)=e^{-2 \pi i \mathrm{tr}_{F/\RR}(x)}$.
\end{itemize}
Indeed, every additive character $\psi:F \to \CC^\times$ is of the form $x \mapsto e^{-2 \pi i\mathrm{tr}_{F/\RR}(cx)}$ for some $c \in F^\times$.  Thus if we prove theorems \ref{thm:gammaneq0} and \ref{thm:gammaeq0} under this assumption we can deduce them in general upon taking a change of variables in $\xi$.  
We assume (A) for the remainder of the section.

\begin{prop} \label{prop:arch:Gaussian}
Let $x \in F^\times$.  The Fourier transform of the distribution 
\begin{align*}
V(F) &\lto \CC\\
w &\longmapsto \psi(xQ(w))
\end{align*}
is 
$$
\gamma(xQ)|\det(xJ)|^{-1/2}\bar{\psi}(x^{-1}Q^{\vee}(w))
$$ 
where 
$\gamma(xQ)=1$ if $F$ is complex and $\gamma(\psi(xQ))=(e^{-\pi i/4})^{a-b}$ if $F$ is real and the signature of $xJ$ is $(a,b)$.
\end{prop}

\begin{proof}
By assumption (A) the self-dual Haar measure on $V(F)=F^n$ is the product of the Lebesgue measure on $F$ if $F=\RR$ and it is the product of twice the Lebesgue measure if $F=\CC$.  

Note that $w \mapsto \psi(xQ(w))$ is a character of the second degree in the sense of \cite{Weil:Certains:groupes}.  Thus the proposition follows upon combining \cite[\S 14, Theorem 2 and \S 26]{Weil:Certains:groupes}. 
\end{proof}

\begin{prop} \label{prop:arch:esti}  Assume that $J$ is a diagonal matrix whose eigenvalues are all $\pm 1$.  
For any $N>0$ if $|\xi| \geq 1$ or $|t| \geq 1$ then
\begin{align*}
\int_{V(F)}\Phi\left(\frac{Q(w)}{t},t,w \right)\psi\left(\frac{\langle \xi,w \rangle}{t} \right)dw \ll \max(|\xi|,1)^{-N}\min(|t|^{n/2},|t|^{-N}).
\end{align*}
Let
\begin{align} \label{a}
a:=(3+\dim_\RR V(F)/2)[F:\RR].
\end{align}
If $\xi \neq 0$ and $|t|<|\xi|^{a}<1$ the integral is 
bounded by a constant
times
\begin{align*}
\frac{|t|^{n/2}}{|\xi|^{n/2+1/[F:\RR]}}.
\end{align*}
\end{prop}

Assuming this proposition (which will be proved in \S \ref{ssec:proof:prop:arch})  we now prove Theorem  \ref{thm:gammaneq0}:

\begin{proof}[Proof of Theorem \ref{thm:gammaneq0}]
Taking an appropriate change of variables we can and do assume that $J$ is a diagonal matrix with $\pm 1$ as eigenvalues.  

Let $D=t\frac{\partial}{\partial t}$ (and, if $F$ is complex, $\bar{D}=\bar{t}\frac{\partial}{\partial \bar{t}}$) viewed as a differential operator on $F^\times$. 
By Proposition \ref{prop:arch:esti} for all $i \in \ZZ_{\geq 0}$ (and $j \in \ZZ_{\geq 0}$ when $F$ is complex) the integral
\begin{align} \label{Desti}
\int_{V(F)}D^i\bar{D}^j\Phi\left(\frac{Q(w)}{t},t,w\right)\psi\left(\frac{\langle \xi,w \rangle}{t} \right)dw
\end{align}
is bounded by a constant depending on $i,j,N$ times $\max(|\xi|,1)^{-N}\min(|t|^{n/2},|t|^{-N})$ if $|\xi| \geq 1$ or $|t| \geq 1$ and bounded 
by a constant times 
$\frac{|t|^{n/2}}{|\xi|^{n/2+1/[F:\RR]}}$ if 
$|t|<|\xi|^{a}<1$.  Here 
$\sigma_1>\mathrm{Re}(s)>\sigma_2$.  

Thus applying integration by parts in $t$ (and $\bar{t}$ if $F$ is complex) if $|\xi|\geq 1$ we obtain a bound of 
\begin{align*}
\mathcal{I}(\Phi,\chi_s)(\xi) &\ll_{N,\sigma_i} \mathcal{C}(\chi_s)^{-N}|\xi|^{-N}\left(\int_{|t|<1}|t|^{n/2+\sigma_2}dt^\times +1\right)\\
& \ll_{\sigma_2} \mathcal{C}(\chi_{s})^{-N}|\xi|^{-N}.
\end{align*}
If $|\xi|<1$ we similarly obtain a bound
\begin{align*}
\mathcal{I}(\Phi,\chi_s)& \ll_{N,\sigma_i} \mathcal{C}(\chi_s)^{-N}|\xi|^{-n/2-1/[F:\RR]}\Bigg(\int_{|t|<|\xi|^{a}}|t|^{n/2+\sigma_2}dt^\times+\int_{|\xi|^{a} \leq |t|<1}|t|^{\sigma_2}dt^\times+1\Bigg)\\
& \ll_{N,\sigma_i}\mathcal{C}(\chi_s)^{-N}|\xi|^{-A}
\end{align*}
for $A>0$ sufficiently large.  
\end{proof}

If $\Phi \in \mathcal{S}(V(F))$ we set notation
for the inverse Fourier transform of $\psi(xQ(w))\widehat{\Phi}(w)$:
\begin{align} \label{inv:FT}
\psi(xQ(D))\Phi(\xi):=\int_{V(F)}
\bar{\psi}(\langle \xi,w \rangle-xQ(w))\widehat{\Phi}(w)dw.
\end{align}
Here the $D$ is present to indicate that
 $\psi(xQ(D))$ can be viewed as a certain formal sum of differential operators, see \cite[Theorem 7.6.2]{Hormander:PDI}, for example.

We denote by 
\begin{align} \label{gamma:factor}
\gamma(\chi_s,\psi)=
\frac{L(1-s,\bar{\chi})\varepsilon(s,\chi,\psi)}{L(s,\chi)}
\end{align}
the usual $\gamma$-factor (not to be confused with Weil's numbers occurring in Proposition \ref{prop:arch:Gaussian}).  
The function is meromorphic as a function of $s$, holomorphic for 
$\mathrm{Re}(s)<1$.  For $\Phi \in \mathcal{S}(F)$ the functional equation for
local zeta function established in Tate's thesis is
\begin{align}\label{Tate:local:fe}
\gamma(\chi_s,\psi)\int_{F^\times}\Phi(t)\chi_s(t)dt^\times=\int_{F^\times} \widehat{\Phi}(t)\bar{\chi}_{1-s}(t)dt^\times.
\end{align}
With this notation out of the way we can now prove Theorem \ref{thm:gammaeq0}:
\begin{proof}[Proof of Theorem \ref{thm:gammaeq0}]

As in the proof of Theorem \ref{thm:gammaneq0} we see that to prove the estimates for $\mathcal{I}(\Phi,\chi_s)(0)$ in the range $\mathrm{Re}(s) \geq -\tfrac{3}{4}$ it is enough to show that $\mathcal{I}(\Phi,\chi_s)(0)$ converges absolutely in that range.  This is our first goal.  

Write $\Phi_1(x,t,w)=\int_F \Phi(y,t,w)\bar{\psi}(xy)dy$.  By Fourier 
inversion we have
\begin{align*}
\mathcal{I}(\Phi,\chi_s)(0)=\int_{F \times F^\times \times V(F)}\Phi_1(x,t,w)
\psi\left( \frac{xQ(w)}{t}\right)\chi_s(t)dt^\times dx dw.
\end{align*}
This converges absolutely for $\mathrm{Re}(s)>0$.  Change variables $x \mapsto tx$ to arrive at
\begin{align} \label{xt:int}
\int_{F \times F^\times} \left(
\int_{ V(F)}\Phi_1(tx,t,w)
\psi(xQ(w)) dw\right)\chi_{s+1}(t)dt^\times dx.
\end{align}
The integral over 
$|x| \leq 1$ converges absolutely for $\mathrm{Re}(s)>-1$.  We now treat the integral over 
$|x|>1$.  Following \cite[Lemma 7.7.3]{Hormander:PDI} one applies the Plancherel formula and Proposition \ref{prop:arch:Gaussian}
to see that this contribution can be written as 
\begin{align}\label{x:large}
\int_{(|x|>1) \times F^\times } \frac{\gamma(xQ)}{|\det xJ|^{1/2}}\bar{\psi}\left(x^{-1}Q^{\vee}(D) \right)\Phi_1(tx,t,w)|_{w=0}\chi_{s+1}(t)dt^\times dx
\end{align}
with notation as in \eqref{inv:FT}.
One has an estimate
\begin{align*} 
|\overline{\psi}( x^{-1}Q^{\vee}(D))\Phi_1(tx,t,w)|_{w=0}-\Phi_1(tx,t,0)| \ll_{\Phi,J}
 \frac{1}{|x|^{1/[F:\RR]}}\sum_{D} |D\Phi_1(tx,t,\cdot)|_2
\end{align*}
from \cite[Lemma 7.7.3, (7.6.7)]{Hormander:PDI}, where $|\cdot|_2:=|\cdot|_{L^2(V(F))}$
and the sum over $D$ is over an $\RR$-vector space basis of the space of invariant differential 
forms on $V(F)$ of degree less than or equal to $a[F:\RR]^{-1}$ with $a$ defined as in \eqref{a}.  
Thus the double integral over $x$ and $t$ in 
\eqref{x:large} is absolutely convergent for $\mathrm{Re}(s)>-1$.  

We now return to \eqref{xt:int}.  We apply \eqref{Tate:local:fe} for each fixed
$x$ to see that the product of $\gamma(\chi_{s+1},\psi)$ and \eqref{xt:int} is equal to 
\begin{align} \label{after:Tate}
\int_{F}\left(\int_{F^\times \times V(F)}\left(
\int_{F}\Phi_1(\alpha x,\alpha,w)\psi(t\alpha)d\alpha \right)\psi(xQ(w))dw\bar{\chi}_{-s}(t)dt^\times\right)dx.
\end{align}
The inner integral is now absolutely convergent for $\mathrm{Re}(s)<0$.  We henceforth
assume that $-n/2<\mathrm{Re}(s)<0$. 

 The integral over $|x| \leq 1$ in \eqref{after:Tate} converges absolutely.  Applying integration by parts as in the proof of Theorem \ref{thm:gammaneq0} one deduces that this contribution is bounded by $O_{N,\sigma_2}(\mathcal{C}(\chi_s)^{-N})$ for any $N>0$ as desired.  
   
 Consider the contribution of $|x|>1$.  
We take a change of variables $\alpha \mapsto x^{-1}\alpha$ 
and then $t \mapsto xt$ to see that this contribution is
\begin{align*}
\int_{|x|>1} \left(\int_{F^\times \times V(F)} 
\left(\int_F \Phi_1(\alpha,x^{-1}\alpha,w)\psi(t\alpha)
d\alpha \right)\psi(xQ(w))dw \bar{\chi}_{-s}(xt)dt^\times \right)
\frac{dx}{|x|}.
\end{align*} 
We claim the integrals over $t$ and $x$ converge absolutely for $0>\mathrm{Re}(s)>\sigma_2$.  Since $\Phi$ was an arbitrary Schwartz function, the claim allows us to apply integration by parts as before and obtain an estimate of $O_{N,\sigma_2}(\mathcal{C}(\chi_s)^{-N})$ on this integral. 

We now prove the absolute convergence claim and thereby complete the proof of the theorem.
We proceed as before, using the Plancherel theorem to write the above as 
\begin{align} \label{to:estimate:0}
\int_{|x|>1}\left(\int_{(F^\times)^2}\frac{\gamma(xQ)}{|xJ|^{n/2}}
\psi(x^{-1}Q^\vee(D))
\int_F\Phi_1(\alpha,x^{-1}\alpha,w)\psi(t\alpha)d\alpha |_{w=0}
\bar{\chi}_{-s}(xt)dt^\times\right)\frac{dx}{|x|}.
\end{align}
Just as before
\begin{align*}
&\left|\psi(x^{-1}Q^{\vee}(D))
\int_F \Phi_1(\alpha,x^{-1}\alpha,w)\psi(t\alpha)d\alpha
|_{w=0}-\int_F\Phi_1(\alpha,x^{-1}\alpha,0)
\psi(t\alpha)d\alpha \right|\\
& \ll_{\Phi,J} \frac{1}{|x|^{1/[F:\RR]}}\sum_D \left|D \left( \int_F \Phi_1(\alpha,x^{-1}\alpha,\cdot)\right)\right|_2
\end{align*}
where the sum over $D$ is as before.  In view of this estimate, it is not hard to see
 that \eqref{to:estimate:0} is absolutely convergent for $-n/2<\mathrm{Re}(s)<0$.
\end{proof}

\subsection{Proof of Proposition \ref{prop:arch:esti}} \label{ssec:proof:prop:arch}
In this section we estimate the oscillatory integral
\begin{align} \label{to:estimate:arch}
\int_{V(F)}\Phi\left(\frac{Q(w)}{t},t,w \right)\psi\left(\frac{\langle \xi,w \rangle}{t} \right)dw
\end{align}
as a function of $t$ and $\xi$ under the assumption that $\xi \neq 0$.  Our goal is to bound this integral by a constant depending on $N_0>0$ and $\Phi$ times
\begin{align} \label{desired}
\max(|\xi|,1)^{-N}\min(|t|^{n/2},|t|^{-N_0})
\end{align}
for any $N_0>0$  provided $|\xi| \geq 1$ or $|t| \geq 1$.  
If $\xi \neq 0$ and $|t|<|\xi|^{a}<1$ we only require the weaker bound
\begin{align*}
\frac{|t|^{n/2}}{|\xi|^{n/2+1/[F:\RR]}}.
\end{align*}
Here $a$ is defined as in \eqref{a}.
The content of Proposition \ref{prop:arch:esti} is that these estimates are valid.

The basic idea is quite simple.  Let
\begin{align} \label{FT1}
\Phi_1(x,y,w):=\int_{F}\Phi(t,y,w) \bar{\psi}(xt)dt
\end{align}
be the inverse Fourier transform of $\Phi$ in the first variable.  Then by Fourier inversion the integral \eqref{to:estimate:arch} is 
\begin{align} \label{FI}
&\int_{F^\times \times V(F)}\Phi_1\left(x,t,w \right)\psi\left(\frac{xQ(w)+\langle \xi,w \rangle}{t} \right)dwdx
\end{align}
Here we have used the fact that $F^\times \subset F$ is of full measure with respect to $dx$.  
We apply the stationary phase method to the integral over $w$.  Provided $x \neq 0$ it has a single nondegenerate critical point at $v=-x^{-1}J^{-1}\xi$ and the norm of the Hessian is $\frac{|\det xJ|}{|t|^{n}}$ everywhere.  At this point we would be done (at least modulo uniformity in $\xi$) if we knew $x$ was bounded away from zero.  In our setting this is not the case, so we have to do a little more work.  

We now begin the formal argument.  We first deal with some trivial cases.  Suppose $|t| \geq \max(|\xi|,1)^{1/2}$.  
In this case we obtain a bound on \eqref{to:estimate:arch} of $O_{\Phi}(|t|^{-N}) \leq O_{\Phi}(\max(|\xi|,1)^{-N/2})$ since $\Phi$ is rapidly decreasing as a function of $t$.  This bound is better than \eqref{desired} for sufficiently large $N$ (under the assumption $|t| \geq \max(|\xi|,1)^{1/2}$).  Thus we can assume that 
\begin{align} \label{first:reduction}
|t|<\max(|\xi|,1)^{1/2}.  
\end{align}
Now assume that $|t| \geq 1$ (so $|\xi| \geq 1$).  Applying integration by parts in $w$ we see that for any $N>0$ the integral \eqref{to:estimate:arch} is 
bounded by $O_{N,\Phi}\left(\max\left(\frac{|t|}{|\xi|},\frac{1}{|\xi|} \right)^N \right)$.  For sufficiently large $N$ this is stronger than the bound \eqref{desired} under the assumption that $|t| \geq 1$.  
Thus we can assume 
\begin{align} \label{t:small}
|t|<1.
\end{align}
Finally, assume that $|\xi| \geq |t|^{-\varepsilon}$ for some $1>\varepsilon>0$, which implies by \eqref{t:small} that $|\xi| \geq 1$.  Applying integration by parts in $w$ we obtain a bound of $O_{\Phi,N}(|\xi|^{-N})$, which is a better bound than \eqref{desired} if $|\xi| \geq |t|^{-\varepsilon}$ if $N$ is sufficiently large.  Thus we can assume that 
\begin{align} \label{eps:bound}
|\xi|<|t|^{-\varepsilon}
\end{align}
for a fixed $1>\varepsilon>0$.

We now begin our application of stationary phase to the integral \eqref{FI}.  We first break the integral into two ranges, one where 
the phase is close to stationary and one where it is not.  To accomplish this, choose nonnegative functions $W_1,W_2 \in C^\infty(V(F))$ such that $W_1+W_2=1$, $W_1$ is identically $1$ in a neighborhood of $0$ and is zero outside of 
\begin{align} \label{half:support}
\{w \in V(F):|w|<\tfrac{1}{2}\}
\end{align}
and $W_2$ vanishes in a neighborhood of zero.  Then \eqref{FI} is equal to the sum over $i \in \{1,2\}$ of 
\begin{align} \label{theis}
\int_{F^\times \times V(F)}W_i\left(\frac{xw+ J^{-1}(\xi)}{|t|^{1/a}}\right)\Phi_1(x,t,w)\psi\left(\frac{xQ(w)+\langle \xi,w \rangle}{t} \right)dxdw.
\end{align}

The power of $|t|$ in the denominator of the argument of $W_i$ will play a role in the estimation of the $i=1$ term.  The $i=2$ term, where the phase is not stationary, can be bounded by a constant depending on $\Phi$ and $N$ times $|t|^{N}$ for any $N>0$ by repeated application of integration by parts in $w$.  In view of the fact that $|t|<1$ and $|\xi|<|t|^{-\varepsilon}$ by \eqref{t:small} and \eqref{eps:bound} we see that the bound \eqref{desired} is valid for the $i=2$ term.

Thus we are reduced to bounding the $i=1$ term.  Taking a change of variables $w \mapsto w-x^{-1}J^{-1}(\xi)$ we see that the $i=1$ case of \eqref{theis} is equal to 
\begin{align*}
\int_{F^\times \times V(F)}W_1\left(\frac{xw}{|t|^{1/a}}\right)\Phi_1\left(x,t,w-\frac{J^{-1}(\xi)
}{x}\right)\psi\left(\frac{xQ(w)}{t}\right)dw\bar{\psi}\left(\frac{Q^{\vee}(\xi)}{xt} \right)dx.
\end{align*}
Apply Plancherel's theorem to the $w$ integral.  Using Proposition \ref{prop:arch:Gaussian} and the notation of 
\eqref{inv:FT} this implies that the above is equal to 
\begin{align} \label{right:before}
\int_{F^\times}\frac{|t|^{n/2}\gamma(xQ)}{|\det xJ|^{1/2}}
\psi\left(-\frac{tQ^\vee (D)}{x}\right)W_1\left(\frac{xw}{|t|^{1/a}} \right)
\Phi_1\left(x,t,w-\frac{J^{-1}(\xi)}{x}\right)|_{w=0}
\bar{\psi}\left(\frac{Q^\vee(\xi)}{xt} \right)dx.
\end{align}
We now employ the estimate
\begin{align*}
&\left|\psi\left(-\frac{tQ^\vee(D)}{x}\right)W_1\left(\frac{xw}{|t|^{1/a}} \right)
\Phi_1\left(x,t,w-\frac{J^{-1}(\xi)}{x}\right)|_{w=0}-
\Phi_1\left(x,t,-\frac{J(\xi)}{x} \right)\right|\\
&\ll \left(\frac{|t|}{|x|} \right)^{1/[F:\RR]}
\sum_{D}\left|D\left( W_1\left( \frac{x(\cdot)}{|t|^{1/a}}\right)\Phi_1\left(x,t,(\cdot)-\frac{J^{-1}(\xi)}{x}\right)\right)\right|_2
\end{align*}
from \cite[(7.6.7)]{Hormander:PDI}, where $|\cdot|_2=|\cdot|_{L^2(V(F)}$ and the 
sum on $D$ is over an $\RR$-vector space basis of the space of invariant 
differential operators on $V(F)$ of degree less than or equal to $a[F:\RR]^{-1}$.  Thus
\eqref{right:before} is bounded by a constant times the sum of 
\begin{align} \label{suff:bound}
\int_{F^\times} \left(\frac{|t|}{|x|}\right)^{n/2}
\left|\Phi_1\left(x,t,-\frac{J^{-1}(\xi)}{x}\right)\right|dx \ll_N |t|^{n/2}
\min(|\xi|,1)^{-n/2}\max(|\xi|,1)^{-N}
\end{align}
(where $N \in \ZZ_{>0}$ is arbitrary) and
\begin{align} \label{more:complicated}
\int_{F^\times} \left(\frac{|t|}{|x|}\right)^{n/2+1/[F:\RR]}
\sum_D \left| D\left(W_1\left( \frac{x(\cdot)}{|t|^{1/a}}\right)
\Phi_1\left(x,t,(\cdot)-\frac{J^{-1}(\xi)}{x}\right) \right)\right|_2dx.
\end{align}
The bound \eqref{suff:bound} is sufficient, so we are left with 
bounding \eqref{more:complicated}.  Since the differential operators occurring in the sum here are of degree less than or equal to $a[F:\RR]^{-1}$
the above is equal to 
\begin{align*}
\int_{F^\times} \left(\frac{|t|}{|x|}\right)^{n/2+1/[F:\RR]}
&\sum_D \left|D\left(W_1\left( \frac{x(\cdot)+J^{-1}(\xi)}{|t|^{1/a}}\right)
\Phi_1(t,x ,\cdot)\right) \right|_2dx\\
&\ll \int_{F^\times} \frac{|t|^{n/2}}{|x|^{n/2+1/[F:\RR]}}
\left| W_3\left( \frac{x(\cdot)+J^{-1}(\xi)}{|t|^{1/a}}\right)\Phi_2(t,x,\cdot)\right|_2dx
\end{align*}
for some $W_3 \in C_c^\infty(V(F))$ supported in \eqref{half:support} and $\Phi_2 \in \mathcal{S}(F \times F \times V(F))$.  

This integrand is supported in the set of $w$ such that 
\begin{align} \label{support}
|xw+J^{-1}(\xi)|<\frac{|t|^{1/a}}{2}.
\end{align}
Suppose that $|\xi| \geq 1$.  Then by \eqref{support} we have
$$
|xw_j| \asymp |\xi|
$$
for some $j$.  Here we have used our assumption that $J$ is diagonal with eigenvalues $\pm 1$.

Using the Schwartz function $\Phi_2$ we obtain an estimate 
on \eqref{more:complicated} of $O_N(|t|^{n/2}|\xi|^{-N})$ 
for $N>0$.  Now assume that $|\xi|<1$ and $|t|<|\xi|^{a}<1$, which, by the assumptions 
of the proposition, is the last case we must treat. Since $|x+1| \leq \tfrac{1}{2}$ implies $|x| \asymp 1$ the bound \eqref{support}
implies $|xw_j| \asymp |\xi_j|$ for some $j$ and we have a bound on \eqref{more:complicated}
of
\begin{align}
\frac{|t|^{n/2}}{|\xi|^{n/2+1/[F:\RR]}}.
\end{align}
\qed

\section{Nonarchimedian case} \label{sec:nonarch:bound}
In this section and the following we assume that $v$ is a nonarchimedian place of $F$ which is omitted from notation: $F:=F_v$.
  We let $q=\OO/\varpi$ where $\varpi$ is a uniformizer of $\OO$. For the remainder of the paper $|\cdot|$ denotes the norm on $F$ such that $|\varpi|=q^{-1}$.

\begin{lem} \label{lem:nonarch:support} Assume that $\chi:F^\times \to \CC^\times$ is a unitary character.
The integrals defining
$\mathcal{I}(\Phi,\chi_s)(\xi)$ converge absolutely if $\mathrm{Re}(s) > 0$. There is an element $\beta \in F^\times$ and an ideal $\mathfrak{f}$ of $\OO$ depending only on $\Phi$ such that 
$\mathcal{I}(\Phi,\chi_s)(\xi)=0$ unless $\xi \in \beta^{-1}V(\OO)$ and the conductor of $\chi$ divides $\mathfrak{f}$.
\end{lem}

\begin{proof}
The absolute convergence statement is clear.  For the remainder of the proof we assume $\mathrm{Re}(s)>0$.  It suffices to prove the vanishing statement under this additional assumption.

For $k>0$ let $\OO^\times(\varpi^k)=1+\varpi^k\OO$.  Since $\Phi \in C_c^\infty(F^2 \times V(F))$ we can and do choose $k$ large enough that 
$$
\Phi(u_1x_1,u_2x_2,u_3w)=\Phi(x_1,x_2,w).
$$
for $u_1,u_2,u_3 \in \OO^\times(\varpi^k)$ and $(x_1,x_2,w) \in F^2 \times V(F)$.
For $u \in \OO^{\times}(\varpi^k)$ consider
$$
\int_{V(F)}\Phi\left(\frac{Q(w)}{ut},ut,w \right)\psi\left(\frac{\langle \xi,w \rangle}{ut} \right)dw.
$$
Taking a change of variables $w \mapsto uw$ we see that this integral is equal to 
$$
\int_{V(F)}\Phi\left(\frac{Q(w)}{t},t,w \right)\psi\left(\frac{\langle \xi,w \rangle}{t} \right)dw.
$$
Taking $\mathfrak{f}=\varpi^k\OO$ the claim that $\mathcal{I}(\Phi,\chi_s)$ vanishes unless the conductor of $\chi$ divides $\mathfrak{f}$ follows.

Let $Y \in V(\OO)$.  For $k$ sufficiently large we have 
\begin{align*}
&\int_{F^{\times } \times V(F)}\Phi\left(\frac{Q(w)}{t},t,w\right)\psi\left(\frac{\langle \xi, w\rangle}{t} \right)dw \chi_s(t)dt^\times\\
&=\int_{F^{\times } \times V(F)}\Phi\left(\frac{Q(w)}{t}-\varpi^{2k}tQ(Y)-\varpi^k\langle Y,Jw\rangle,t,w-t\varpi^kY\right)\psi\left(\frac{\langle \xi, w\rangle}{t} \right)dw \chi_s(t)dt^\times.
\end{align*}
Taking a change of variable $w \mapsto w+t \varpi^kY$ we see that this is equal to 
\begin{align*}
\int_{F^{\times } \times V(F)}\Phi\left(\frac{Q(w)}{t},t,w\right)\psi\left(\frac{\langle \xi, w\rangle}{t} +\varpi^k \langle \xi,Y \rangle\right)dw \chi_s(t)dt^\times.
\end{align*}
The claim that $\mathcal{I}(\Phi,\chi_s)$ is supported in $\beta^{-1}V(\OO)$ for some $\beta$ depending only on $\Phi$ follows.

\end{proof}

The following is the main theorem of this section:

\begin{thm} \label{thm:nonarch:bound}
The quotient
$$
\frac{\mathcal{I}(\Phi,\chi_s)(\xi)}{L(s+1,\chi)^{\delta_\xi}}
$$
admits a holomorphic continuation to $\mathrm{Re}(s)>-n/2$.  For 
$$
\sigma_1>\mathrm{Re}(s) >\sigma_2>-n/2
$$ 
it is bounded by a constant depending on $\sigma_1,\sigma_2,\Phi$ times
\begin{align*}
\max(1,(1-\delta_\xi)|\xi|^{1-n/2}).
\end{align*}
\end{thm}
\noindent Here if $\xi=0$ we interpret $(1-\delta_\xi)|\xi|^{1-n/2}$ as meaning $0$.

  We begin with a preliminary proposition:
\begin{prop} \label{prop:reduction}
Assume $\psi$ is unramified and that $J$ is a diagonal matrix in $\gl_n(\OO)$.  Assume moreover that $c,k \in \ZZ_{>0}$, $c>k$, $\Psi \in C^\infty(V(\OO))$ and that $\beta =0$ or $\beta \in \OO^\times$.  As a function of $s$
\begin{align*}
L(s+1,\chi)^{-\delta_\xi}\int_{|t|<q^{-c}}\int_{F^\times \times V(F)}
\one_{\varpi^{k}\OO}(x-\beta_1)\Psi(w)\psi\left(\frac{xQ(w)+\langle \xi,w \rangle}{t} \right)dwdx\chi_s(t)dt^\times
\end{align*}
admits a holomorphic continuation to $\mathrm{Re}(s)>-n/2$ that is bounded by a constant depending on $\sigma_1,\sigma_2,c,\beta,k$ times 
$$
\max(1,(1-\delta_\xi)|\xi|^{1-n/2})
$$
provided that $\sigma_1>\mathrm{Re}(s) >\sigma_2>-n/2$.
\end{prop}

\begin{proof}

Throughout the proof we assume that 
\begin{align} \label{sig:bound}
\sigma_1>\mathrm{Re}(s)>\sigma_2>-n/2.
\end{align}

We start with the integral 
\begin{align} \label{int:to:bound}
\int_{|t|<q^{-c}}\int_{F^\times \times V(F)}\Psi(w)
\one_{\varpi^{k}\OO}(x-\beta)\psi\left(\frac{xQ(w)+\langle \xi,w \rangle}{t} \right)dwdx\chi_s(t)dt^\times.
\end{align}  
We will often use the fact that for $t \in \OO \cap F^\times$ the Fourier transform of the distribution $w \mapsto \psi\left(\frac{Q(w)}{t} \right)$ on $V(F)$ is 
\begin{align} \label{FT}
\gamma(t^{-1}Q)\frac{|t|^{n/2}}{\left|\det J\right|^{1/2}}\bar{\psi}\left(tQ^{\vee}(w) \right)
\end{align}
where $|\gamma(t^{-1}Q)|=1$ \cite[\S 14]{Weil:Certains:groupes}.

Assume $\beta \in \OO^\times$.  Then by the Fourier transform computation just mentioned we have that \eqref{int:to:bound} is equal to 
\begin{align*}
|\det J|^{-1/2}&\int_{|t|<q^{-c}}\int_{F^\times \times V(F)}\widehat{\Psi}\left(\frac{\xi}{t}- w\right)
\one_{\varpi^{k}\OO}(x-\beta)\gamma(xt^{-1}Q)\bar{\psi}\left(\frac{tQ(w)}{x} \right)dwdx\chi_{s+n/2}(t)dt^\times\\
&\ll_{\Psi} \zeta(\mathrm{Re}(s)+n/2) \ll_{\sigma_1,\sigma_2}1.
\end{align*}

Assume now that $\beta=0$.  We then write \eqref{int:to:bound} as $q^k$ times 
\begin{align}
&\int_{|t|<q^{-c}}\int_{F^\times \times V(F)}\Psi(w)\psi\left(\frac{\varpi^k xQ(w)+\langle \xi,w \rangle}{t} \right)dwdx\chi_s(t)dt^\times \nonumber \\
&=\chi_s(\varpi^k)\int_{|t|<q^{-c+k}}\int_{F^\times \times V(F)}\Psi(w)
\psi\left(\frac{ xQ(w)+\langle \varpi^{-k}\xi,w \rangle}{t} \right)dwdx\chi_s(t)dt^\times \nonumber\\
&=\chi_s(\varpi^k)\int_{|t|<q^{-c+k}}\int_{ V(F)} \Psi(w) \sum_{i=0}^{v(t)}\sum_{x \in (\OO/t\varpi^{-i})^\times}\psi\left(\frac{ \varpi^ixQ(w)+\langle \varpi^{-k}\xi,w \rangle}{t} \right)dw\chi_{s+1}(t)dt^\times \nonumber\\
&=\chi_s(\varpi^k)\sum_{i=0}^\infty\chi_{s+1}(\varpi^{i}) \int_{|t|<q^{-c+k}}\int_{ V(F)}\Psi(w)
\sum_{x \in (\OO/t)^\times}\psi\left(\frac{ xQ(w)+\langle \varpi^{-i-k}\xi,w \rangle}{t} \right)dw\chi_{s+1}(t)dt^\times. \nonumber
\end{align}
The $i$ summand here vanishes unless $\xi \in \varpi^{i+a}V(\OO)$ for some $a \in \ZZ$ depending only on $k$ and $\Psi$.

We now use the Fourier transformation computation \eqref{FT} to write this as 
\begin{align} \label{all}
&\chi_s(\varpi^k)\sum_{i=0}^\infty\chi_{s+1}(\varpi^{i}) \int_{|t|<q^{-c+k}}\sum_{x \in (\OO/t)^\times}\frac{\gamma(xt^{-1}Q)}{|\det J|^{1/2}}\\& \times \int_{ V(F)}\widehat{\Psi}\left(\frac{\xi}{\varpi^{i+k}t}-w \right)\bar{\psi}\left(\frac{tQ^{\vee}(w)}{x} \right)
dw\chi_{s+1+n/2}(t)dt^\times. \nonumber 
\end{align}
We note that
\begin{align} \label{inside}
&\int_{|t|<q^{-c+k}}
\sum_{x \in (\OO/t)^\times}\frac{\gamma(xt^{-1}Q)}{|\det J|^{1/2}}\int_{ V(F)}\widehat{\Psi}\left(\frac{\xi}{\varpi^{i+k}t}-w \right)\bar{\psi}\left(\frac{tQ^{\vee}(w)}{x} \right)
dw\chi_{s+1+n/2}(t)dt^\times\\
&\ll_{J,\Psi,c,k} \zeta(\mathrm{Re}(s)+n/2) \ll_{\sigma_1,\sigma_2} \nonumber 1.
\end{align}
  Since the $i$th summand in \eqref{all} vanishes unless $\xi \in \varpi^{i+a}V(\OO)$ we deduce that if $\xi \neq 0$ then \eqref{all} is bounded by a constant depending on $c$, $k$, $\sigma_1$, $\sigma_2$, $\Psi$ times 
$$
\max(1,|\xi|^{1-n/2})
$$
if \eqref{sig:bound} is valid.  This yields the theorem in this case.

Now assume that $\xi=0$.  Then \eqref{all} is equal to 
\begin{align*}
\chi_s(\varpi^k)L(s+1,\chi)\int_{|t|<q^{-c+k}}
\sum_{x \in (\OO/t)^\times}\frac{\gamma(xt^{-1}Q)}{|\det J|^{1/2}}\int_{ V(F)}\widehat{\Psi}\left(-w \right)\bar{\psi}\left(\frac{tQ^{\vee}(w)}{x} \right)
dw\chi_{s+1+n/2}(t)dt^\times. 
\end{align*}
This divided by $L(s+1,\chi)$ is bounded by a constant depending on $\Psi,\sigma_1,\sigma_2$ provided that \eqref{sig:bound} is valid.
\end{proof}

\begin{proof}[Proof of Theorem \ref{thm:nonarch:bound}]
 Temporarily write
$\mathcal{I}(\psi)(\xi)$ for $\mathcal{I}(\Phi,\chi_s)(\xi)$ defined with respect to $\psi$, and  write $\psi_c(x):=\psi(cx)$ for $c \in F^\times$.  One has
$$
\mathcal{I}(\psi_c)(\xi)=\mathcal{I}(\psi)(c\xi).
$$
It therefore suffices to prove the theorem in the special case where $\psi$ is unramified, which we henceforth assume.

Upon taking a change of variables in $w$ we can assume that $J$ is diagonal \cite[Theorem 3.5]{Scharlau:Quadratic:Hermitian}. This also entails replacing $\xi$ by $A\xi$ for some $A \in \GL_n(F)$, but this is harmless.   Absorbing powers of $\varpi$ into $\Phi$ we can also assume that $J \in \gl_n(\OO)$. 

Let $\Phi_1(x,y,w):=\int_F \Phi(t,y,w)\bar{\psi}(xt)dt$ be the inverse Fourier transform of $\Phi$ in the first variable.  By Fourier inversion we then have
\begin{align*}
\mathcal{I}(\Phi,\chi_s)(\xi)=\int_{F \times F^\times \times V(F)} \Phi_1(x,t,w)\psi\left(\frac{xQ(w)+\langle \xi, w \rangle}{t}\right)\chi_s(t)dt^\times dxdw.
\end{align*}
We can and do assume that $\Phi_1$ is equal to 
$$
\one_{(\beta_1,\beta_2,\mu)\varpi^{-a}+\varpi^k(\OO^2 \times V(\OO))}
$$
with $a \in \ZZ_{ \geq 0}$, $\beta_1,\beta_2 \in \OO$, $\mu \in V(\OO)$, and $k \geq 0$
since every element of $C_c^\infty(F^2 \times V(F))$ is a finite sum of functions of this form.  If $\beta_2\varpi^{-a} \not \in \varpi^k\OO$ then the integral over $t$ in $\mathcal{I}(\Phi,\chi_s)(\xi)$ is compactly supported.  In this case the bound asserted by Theorem \ref{thm:nonarch:bound} is trivial to obtain.  We therefore can and do assume that $\beta_2 =0$.  

 We then have
\begin{align*}
\mathcal{I}(\Phi,\chi_s)(\xi)
&=\int_{\varpi^k\OO}\int_{F^\times \times V(F)}
\one_{\OO \times V(\OO)}\left(\frac{x-\beta_1\varpi^{-a}}{\varpi^k},\frac{w-\mu\varpi^{-a}}{\varpi^{k}}\right)\\& \times \psi\left(\frac{xQ(w)+\langle \xi,w \rangle}{t} \right)dwdx\chi_s(t)dt^\times.
\end{align*}

Taking a change of variables $(x,t,w) \mapsto 
(\varpi^{-a}x,\varpi^kt,\varpi^{-a}w)$ we see that this is equal to 
$\chi_s(\varpi^k)q^{an +a}$ times
\begin{align*}
\int_{\OO}\int_{F^\times \times V(F)}
\one_{\OO \times V(\OO)}\left(\frac{x-\beta_1}{\varpi^{k+a}},\frac{w-\mu}{\varpi^{k+a}}\right)
\psi\left(\frac{xQ(w)+\langle \varpi^{2a+k}\xi,w \rangle}{\varpi^{3a+k}t} \right)dwdx
\chi_s(t)dt^\times.
\end{align*}
The contribution of $|t| > q^{-c}$ is bounded
by a constant depending on $c$, $k$, $a$, and $\sigma_1,\sigma_2$.  Thus it suffices to bound, for any $\Psi \in C_c^\infty(V(\OO))$, the integral
\begin{align*}
\int_{|t|<q^{-c}}\int_{F^\times \times V(F)}
\one_{\varpi^{k}\OO}(x-\beta_1)\Psi(w)\psi\left(\frac{xQ(w)+\langle \xi,w \rangle}{t} \right)dwdx\chi_s(t)dt^\times
\end{align*}
for arbitrary $k>0$ for any fixed $c$ (which is allowed to depend on $k$ and $\sigma_1,\sigma_2$).  

Finally we check that it is enough to consider the case where $\beta \in \OO^\times$ or $\beta =0$.  Assume that $\beta \in \varpi^i\OO$ with $1 \leq i <k$.  We 
can and do assume that $i<c$.
Then taking a change of variables 
$x \mapsto \varpi^ix$ and $t \mapsto \varpi^it$ 
we see that the integral above is equal to $\chi_s(\varpi^i)q^i$ times
\begin{align*}
\int_{|t|<q^{-c+i}}\int_{F^\times \times V(F)}\Psi(w)
\one_{\varpi^{k-i}\OO}(x-\beta_1\varpi^{-i})\psi\left(\frac{xQ(w)+\langle \varpi^{-i}\xi,w \rangle}{t} \right)dwdx\chi_s(t)dt^\times.
\end{align*}
We rename variables $c \mapsto c-i$, $k \mapsto k+i$ and $\xi \mapsto \varpi^{-i}\xi$ and apply Proposition \ref{prop:reduction} to deduce the theorem.
\end{proof}

\section{The unramified computation}

\label{sec:unram:comp}

We work locally in this section at a nonarchimedian, non-dyadic place $v$ which is omitted from notation. We assume that $F$ is absolutely unramified at $v$, that $\psi$ is unramified, and that $J \in \GL_n(\OO)$, where $J$ is the matrix of $Q$.
  We assume that the Haar measure on $V(F)$ is normalized so that $V(\OO)$ has measure $1$.  This is the self-dual Haar measure with respect to the pairing $(w_1,w_2) \mapsto \psi(\langle w_1,w_2 \rangle)$ since $\psi$ is unramified.  The main theorem of this section, Theorem \ref{thm:nonarch:comp}, computes $\mathcal{I}(\one_{\OO^2 \times V(\OO)},\chi_s)$.  To prove it we first prove a series of lemmas.

\begin{lem} \label{lem:nondyadic}
 Let $u \in \OO^\times$, $t \in \OO$ and $\xi_0 \in F$.  One has
\begin{align}
\int_{\OO}\psi\left(\frac{u y^2+\xi_0y}{t} \right)dy=\one_{\OO}(\xi_0) \psi\left(\frac{-\xi_0^2}{4ut} \right)\int_{\OO} \psi\left(\frac{uy^2}{t} \right)dy.
\end{align}
The integral
$\int_{\OO} \psi\left(\frac{uy^2}{t} \right)dy$ has complex norm $|t|^{1/2}$ and satisfies
\begin{align*}
\int_{\OO} \psi\left(\frac{uy^2}{t} \right)dy=\left(\frac{u}{\varpi} \right)^{v(t)}
\int_{\OO} \psi\left(\frac{y^2}{t} \right)dy
\end{align*}
where $\left(\frac{\cdot}{\varpi} \right)$ is the Legendre character.  
\end{lem}
\begin{proof}
 One has
\begin{align} \label{FT2}
\int_{\OO} \psi\left(\frac{uy^2 +\xi_0y}{t} \right)dy= \psi\left(\frac{-\xi_0^2}{4ut} \right)\int_{\OO} \psi\left(\frac{uy^2}{t} \right)dy
\end{align}
by a change of variables $y \mapsto y-\frac{\xi_0}{2u}$.
To compute the latter integral assume first that $v(t)=1$.  Then 
\begin{align}
\int_{\OO}\psi\left(\frac{uy^2}{t} \right)dy&=\int_{\OO^\times}\psi\left(\frac{uy^2}{t} \right)dy+q^{-1}\\
\nonumber &=q^{-1}+\int_{\OO^{\times}}\left(\left(\frac{y}{\varpi} \right) +1\right)\psi\left(\frac{uy}{t} \right)dy\\
\nonumber &=\int_{\OO^{\times}}\left(\frac{y}{\varpi} \right)\psi\left(\frac{uy}{t} \right)dy.
\end{align}
From this formula (and a change of variables $y \mapsto u^{-1}y$) it is evident that 
$$
\int_{\OO}\psi\left(\frac{uy^2}{t} \right)dy=\left(\frac{u}{\varpi} \right)\int_{\OO}\psi\left(\frac{y^2}{t} \right)dy.
$$
It is well-known that $|\int_{\OO}\psi\left(\frac{y^2}{t} \right)dy|_{\mathrm{st}}=|t|^{1/2}$ (see \cite[p. 147]{Ireland:Rosen}, for example) so the lemma is proven in the $v(t)=1$ case.

For $v(t)>1$ one has
\begin{align*}
&\int_{\OO}\psi\left(\frac{uy^2}{t} \right)dy=\int_{\OO^\times}\psi\left(\frac{uy^2}{t} \right)dy+\int_{\varpi \OO}\psi\left(\frac{uy^2}{t} \right)dy\\
&=|t|\sum_{\alpha \in (\OO/t \varpi^{-1})^\times}\sum_{x \in \OO/\varpi} \psi\left(\frac{u(\alpha^2+2\alpha t\varpi^{-1} x)}{t} \right)+q^{-1}\int_{\OO}\psi\left(\frac{uy^{2}}{t\varpi^{-2}} \right)dy\\
&=q^{-1}\int_{\OO}\psi\left(\frac{uy^{2}}{t\varpi^{-2}} \right)dy.
\end{align*}
Thus by induction
\begin{align} \label{by:ind}
\int_{\OO} \psi\left(\frac{uy^2}{t} \right)dy=\begin{cases}|t|^{1/2} &\textrm{ if }2|v(t)\\
(|t|q)^{1/2}\int_{\OO}\psi\left(\frac{uy^{2}}{t\varpi^{-(v(t)-1)}} \right) dy&\textrm{ if }2 \nmid v(t).\end{cases}
\end{align}
Given what we have already proven the lemma follows for $v(t)>1$ as well.
\end{proof}

For the remainder of this section we make the following important simplifying assumption:
\begin{itemize}
\item[(A)] The dimension $n$ is even.
\end{itemize}

Viewing $w \mapsto \psi(t^{-1}Q(w))$ as a character of second degree on $V(F)$ in the sense of \cite{Weil:Certains:groupes}, we can define (as in loc.~cit.) the factor
$$
\gamma(t^{-1}Q):=\gamma(\psi(t^{-1}Q(w))).
$$

\begin{lem} \label{lem:nonarch:Gaussian}  
For $t \in \OO\cap F^\times$ one has
\begin{align*}
\int_{V(\OO)}\psi\left(\frac{Q(w)}{t}\right)dw=|t|^{n/2}
\gamma(t^{-1}Q).
\end{align*}
The factor $\gamma(t^{-1}Q)$ depends only on the valuation of $t$, and hence can be viewed as an unramified character of $F^\times$.  More precisely 
$$
\gamma(t^{-1}Q)=\left(\frac{(-1)^{n/2}\det Q}{\varpi} \right)^{v(t)}
$$
where $\left(\frac{\cdot}{\varpi} \right)$ is the Legendre symbol.
\end{lem}

\begin{proof} 

The first assertion of the lemma is a consequence of \cite[\S 14, Th\'eor\`eme 2]{Weil:Certains:groupes}. Indeed, loc.~cit.~is the statement that the Fourier transform of the distribution  $w \mapsto \psi\left(\frac{Q(w)}{t}\right)$ is $|t|^{n/2}\gamma(t^{-1}Q)\bar{\psi}\left(tQ^{\vee}(w) \right)$.  Here $|t|^{n}|\det J^{-1}|=|t|^{n}=|\rho|^{-1}$ in the notation of loc.~cit.~since $dv$ is the self-dual Haar measure.  Thus using the fact that the Fourier transform sends products to convolutions and $\widehat{\one}_{V(\OO_F)}=\one_{V(\OO_F)}$ we have
\begin{align*}
\int_{V(F)}\one_{V(\OO)}(w)\psi\left(\frac{Q(w)}{t}\right)dw&=|t|^{n/2}\gamma(t^{-1}Q)\int_{V(F)}\one_{V(\OO_F)}(0-w)\bar{\psi}\left(tQ^\vee(w) \right)dw\\
&=|t|^{n/2}\gamma(t^{-1}Q).
\end{align*}
We now use this identity to prove that $\gamma(t^{-1}Q)$ depends only on the valuation of $t$.

For any $g \in \GL_n(\OO)$ one has 
\begin{align} \label{J:act}
\int_{V(\OO)} \psi\left(\frac{ Q(w)}{t} \right)dw
=\int_{V(\OO)} \psi\left(\frac{ \langle gw,gw\rangle_Q}{t} \right)dw.
\end{align}
Unimodular quadratic forms over non-dyadic discrete valuation rings can be diagonalized \cite[\S 92]{OMeara:Intro:quad:forms}.  Thus we can choose $g \in \GL_n(\OO)$ such that \eqref{J:act} is equal to 
\begin{align*}
\prod_{i=1}^{n} \int_{\OO}\psi\left(\frac{a_ix^2}{t} \right)dx
\end{align*} 
for some $a_1,\dots,a_{n} \in \OO^\times$.
By Lemma \ref{lem:nondyadic} replacing $t$ by $u^{-1}t$ for $u \in \OO^\times$ has the effect of replacing this by 
\begin{align}
&\prod_{i=1}^{n} \int_{\OO}\psi\left(\frac{ua_ix^2}{t} \right)dx=\prod_{i=1}^{n}\left(\frac{u}{\varpi} \right)^{v(t)}\int_{\OO}\psi\left(\frac{a_ix^2}{t} \right)dx
=\prod_{i=1}^{n}\int_{\OO}\psi\left(\frac{a_ix^2}{t} \right)dx. \label{diag}
\end{align}
Here we have used the fact that $n$ is even.  The assertion that $\gamma(t^{-1}Q)$ depends only on the valuation of $t$ follows.

Our last task is to compute $\gamma(t^{-1}Q)$.  Let $p$ be the residual characteristic of $F$ and let
\begin{align*}
\psi_{\QQ_p}(x)&=e^{2 \pi i \mathrm{pr}(x)}
\end{align*}
where $\mathrm{pr}(x) \in \ZZ[p^{-1}]$ is chosen so that $\mathrm{pr}(x)-x \in \ZZ_p$.   Let
$\psi_{F}:=\psi_{\QQ_{p}} \circ \mathrm{tr}_{F/\QQ_{p}}$.  
There is a $c \in F^\times$ such that 
$\psi(x)=\psi_F(cx)$ for all $x \in F$.  Since $F$ is absolutely unramified and $\psi$ is unramified $c \in \OO^\times$.
Since $\gamma(t^{-1}Q)$ only depends on the valuation of $t$, it follows that replacing $\psi$ by $\psi_F$ will not change the value of 
$\gamma(t^{-1}Q)$.  Thus to compute $\gamma(t^{-1}Q)$ we can and do assume $\psi=\psi_F$.

Since \eqref{J:act} is equal to \eqref{diag} we can apply Lemma \ref{lem:nondyadic} to see that 
$$
\gamma(t^{-1}Q)=|t|^{-n/2}\left(\frac{\det Q}{\varpi} \right)^{v(t)}\left(\int_{\OO}\psi_F\left(\frac{y^2}{t} \right)dy\right)^{n}.
$$
Applying \eqref{by:ind}, if $v(t)$ is even this is $1$.  If $v(t)$ is odd it is equal to 
\begin{align*}
\left(\frac{\det J}{\varpi} \right)q^{-\dim V/2}\left(\int_{\OO}\psi_F\left(\frac{y^2}{t\varpi^{-(v(t)-1)}} \right)dy\right)^{n}.
\end{align*}
We now use \cite[Theorem 5.2]{Szechtman:Gauss}
to deduce that
\begin{align*}
q^{\dim V}\left(\int_{\OO}\psi_F\left(\frac{y^2}{t\varpi^{-(v(t)-1)}} \right)dy\right)^{n}
&=\left(\sum_{x \in \ZZ/p\ZZ}\psi_{\QQ_p}\left(\frac{x^2}{p}\right)\right)^{[F:\QQ_p]n}\\
&=\left(p\left(\frac{-1}{p} \right)\right)^{[F:\QQ_p]n/2}\\
&=\left(q\left(\frac{-1}{\varpi} \right)\right)^{n/2}.
\end{align*}
Here in the penultimate equality we have used Gauss' theorem on the sign of quadratic Gauss sums (see, e.g.~\cite[\S 6.4]{Ireland:Rosen}).
The lemma follows.

\end{proof}

\begin{lem} \label{lem:prelim:comp} Suppose that $x \in \OO^\times$, $t \in \OO \cap F^\times$, and $\xi \in V(F)$.   One has
$$
\int_{V(\OO_{F})} 
\psi\left(\frac{xQ(w)+\langle\xi,w \rangle}{t} \right)dw=\one_{V(\OO)}(\xi)\bar{\psi}\left(\frac{Q^\vee (\xi)}{xt}\right)|t|^{n/2}\gamma(t^{-1}Q).
$$
\end{lem}

\begin{proof} 
Since $J \in \GL_n(\OO)$ it is not hard to see that the integral vanishes identically unless $\xi \in V(\OO)$.  We henceforth assume that $\xi \in V(\OO)$.

Taking a change of variables $w \mapsto w-\frac{J^{-1}\xi}{x}$ we see that
\begin{align*}
&\int_{V(\OO_{F})} 
\psi\left(\frac{xQ(w)+\langle\xi,w \rangle}{t} \right)dw=\bar{\psi}\left( \frac{Q(J^{-1}\xi)}{xt}\right)\int_{V(\OO_{F})} 
\psi\left(\frac{xQ(w)}{t} \right)dw.
\end{align*}
Invoking Lemma \ref{lem:nonarch:Gaussian} we deduce the current lemma.

\end{proof}

To ease notation let
\begin{align}
\label{Gauss}
\mathcal{G}:F^{\times} \lto \CC^\times
\end{align}
be the unique unramified character that agrees with $t \mapsto \gamma(t^{-1}Q)$ when restricted to $\OO \cap F^\times$.

\begin{thm} \label{thm:nonarch:comp} Let $\chi:F^\times \to \CC^\times$ be a character and let $\mathrm{Re}(s)>0$.  If $\chi$ is ramified then 
$$
\mathcal{I}(\one_{\OO^2\times V(\OO)},\chi_s)=0.
$$
If $\chi$ is unramified then
\begin{align*}
&\mathcal{I}(\one_{\OO^2\times V(\OO)},\chi_s)(\xi)L(s+n/2+1,\chi\mathcal{G})\\
&=\sum_{j=0}^\infty \chi_{s+1}(\varpi^j)\one_{V(\OO)}(\varpi^{-j}\xi)\int_{\OO}\one_{t\OO}(Q^{\vee}(\xi))\chi\mathcal{G}_{s+n/2}(t)dt^\times.
\end{align*}

\end{thm}

\begin{proof}
One has
\begin{align*}
\mathcal{I}(\one_{\OO^2\times V(\OO)},\chi_s)(\xi)&:=\int_{F^{\times} \times V(F)}\one_{\OO^2 \times V(\OO_{F})}\left(\frac{Q(w)}{t},t,w\right)\psi\left(\frac{\langle \xi,w \rangle}{t}\right)dw\chi_s(t)dt^\times.
\end{align*}
Via an easy analogue of the argument of Lemma \ref{lem:nonarch:support} one sees that this integral vanishes identically unless $\chi$ is unramified.  We therefore assume for the rest of the proof that $\chi$ is unramified. 

The integral above is equal to 
\begin{align*}
&\int_{\OO}\int_{V(\OO)}\sum_{x \in \OO/t}
\psi\left(\frac{xQ(w)+\langle\xi,w \rangle}{t}\right) dw\chi_{s+1}(t)dt^\times\\
&=\int_{\OO \times V(\OO)}\sum_{j=0}^{v(t)}\sum_{x \in (\OO/\varpi^{-j}t)^\times}\psi\left(\frac{x\varpi^jQ(v)+\langle\xi,v \rangle}{t}\right) dw\chi_{s+1}(t)dt\\
&=\sum_{j=0}^\infty \chi_{s+1}(\varpi^j)\int_{\OO \times V(\OO)}\sum_{x \in (\OO/t)^\times}\psi\left(\frac{xQ(w)+
\langle \varpi^{-j}\xi,w\rangle
}{t}\right) dw\chi_{s+1}(t)dt^\times.
\end{align*}
We now employ Lemma \ref{lem:prelim:comp} to write this as
\begin{align*}
&\sum_{j=0}^\infty \chi_{s+1}(\varpi^j)\one_{V(\OO)}(\varpi^{-j}\xi)\int_{\OO}\sum_{x \in (\OO/t)^\times}\bar{\psi}\left(\frac{Q^{\vee}(\xi)}{xt} \right)\chi_{s+1+n/2}(t)\mathcal{G}(t)dt^\times\\
&=\sum_{j=0}^\infty \chi_{s+1}(\varpi^j)\one_{V(\OO)}(\varpi^{-j}\xi)\int_{\OO}\left(
\one_{t\OO}(Q^\vee(\xi))-q^{-1}\one_{\varpi\OO}(t)\one_{t\varpi^{-1}}(Q^\vee(\xi))\right)
\chi\mathcal{G}_{s+n/2}(t)dt^\times\\
&=\sum_{j=0}^\infty \chi_{s+1}(\varpi^j)\one_{V(\OO)}(\varpi^{-j}\xi)(1-\chi\mathcal{G}_{s+n/2+1}(\varpi))\int_{\OO}\one_{t\OO}(Q^\vee(\xi))\chi\mathcal{G}_{s+n/2}(t)dt^\times.
\end{align*}
\end{proof}


\bibliography{../refs}{}
\bibliographystyle{alpha}

\end{document}